\documentclass[12pt,a4paper,reqno]{amsart}
\usepackage[a4paper,top=3cm,bottom=3cm,inner=3cm,outer=3cm,includehead,includefoot]{geometry}
\usepackage{amsmath,amssymb,amsthm,wasysym,calc,verbatim,enumitem,tikz,url,hyperref,mathrsfs,cite}
\usetikzlibrary{shapes.misc,calc,intersections,patterns,decorations.pathreplacing}
\hypersetup{colorlinks=true,citecolor=blue,linktoc=page}

\newtheorem{theorem}{Theorem}
\newtheorem{lemma}[theorem]{Lemma}
\newtheorem{prop}[theorem]{Proposition}
\newtheorem{claim}[theorem]{Claim}

\newtheorem{obs}[theorem]{Observation}
\newtheorem{problem}[theorem]{Problem}
\theoremstyle{definition}
\newtheorem{definition}[theorem]{Definition}
\newtheorem*{spanning}{The spanning algorithm}

\newenvironment{clmproof}[1]{\begin{proof}[Proof of Claim~\ref{#1}]\let\qednow\qedsymbol\renewcommand{\qedsymbol}{}}{\qednow \end{proof}}

\numberwithin{theorem}{section}
\setlist[itemize]{leftmargin=1cm}
\setlist[enumerate]{leftmargin=1.5cm}

\everydisplay{\thickmuskip=7mu}
\renewcommand{\leq}{\leqslant}
\renewcommand{\geq}{\geqslant}
\renewcommand{\le}{\leqslant}
\renewcommand{\ge}{\geqslant}
\renewcommand{\to}{\rightarrow}

\let\epsilon\varepsilon
\let\eps\varepsilon

\def\H{\mathbb{H}}
\def\N{\mathbb{N}}
\def\P{\mathbb{P}}

\def\R{\mathbb{R}}
\def\Z{\mathbb{Z}}
\def\A{\mathcal{A}}
\def\C{\mathcal{C}}
\def\D{\mathcal{D}}
\def\EE{\mathcal{E}}
\def\F{\mathcal{F}}

\def\K{\mathcal{K}}
\def\n{\mathcal{N}}
\def\PP{\mathcal{P}}
\def\S{\mathcal{S}}

\def\U{\mathcal{U}}

\def\stab{\mathcal{S}}
\def\hier{\mathcal{H}}
\def\root{v_{\text{root}}}
\def\ispan{I^\times}
\def\<{\langle}
\def\>{\rangle}
\def\0{\mathbf{0}}
\def\ih{\mathrm{IH}}

\def\edge{\partial}
\def\top{\mathrm{top}}
\def\bottom{\mathrm{bottom}}

\title{The sharp threshold for the Duarte model}

\author[B. Bollob\'as \and H. Duminil-Copin \and R. Morris \and P.J. Smith]{B\'ela Bollob\'as \and Hugo Duminil-Copin \and Robert Morris \and Paul Smith}

\address{Department of Pure Mathematics and Mathematical Statistics, Wilberforce Road, Cambridge, CB3 0WA, UK, and Department of Mathematical Sciences, University of Memphis, Memphis, TN 38152, USA, and London Institute for Mathematical Sciences, 35a South Street, London, W1K 2XF, UK}
\email{b.bollobas@dpmms.cam.ac.uk}
\address{D\'epartement de Math\'ematiques, Universit\'e de Gen\`eve, 2-4 Rue du Li\`evre, Gen\`eve, Switzerland}
\email{hugo.duminil@unige.ch}
\address{IMPA, Estrada Dona Castorina 110, Jardim Bot\^anico, Rio de Janeiro, 22460-320, Brazil}
\email{rob@impa.br}
\address{Department of Pure Mathematics and Mathematical Statistics, Wilberforce Road, Cambridge, CB3 0WA, UK}
\email{p.j.smith@dpmms.cam.ac.uk}
\thanks{B.B.\ is partially supported by NSF grant DMS~1301614 and MULTIPLEX grant no.~317532, H.D.\ by a grant from the Swiss FNS and the NCCR SwissMap (also funded by the Swiss FNS), R.M.\ by CNPq (Proc.~479032/2012-2 and Proc.~303275/2013-8), and P.S.\ by a CNPq bolsa PDJ}

\date{\today}


\begin{document}

\begin{abstract}
The class of critical bootstrap percolation models in two dimensions was recently introduced by Bollob\'as, Smith and Uzzell, and the critical threshold for percolation was determined up to a constant factor for all such models by the authors of this paper. Here we develop and refine the techniques introduced in that paper in order to determine a sharp threshold for the Duarte model. This resolves a question of Mountford from 1995, and is the first result of its type for a model with drift.
\end{abstract}

\maketitle

%\tableofcontents

%\newpage

\section{Introduction}

In this paper we resolve a 20 year old problem of Mountford~\cite{Mount} by determining the sharp threshold for a particular monotone cellular automaton related to the classical $2$-neighbour bootstrap percolation model. This model was first studied by Duarte~\cite{Duarte}, and is the most fundamental model for which a sharp threshold had not yet been determined. Indeed, our main theorem is the first  result of its type for a critical bootstrap model that exhibits `drift', and is an important step towards a complete understanding of sharp thresholds in two-dimensional bootstrap percolation. 

We will begin by stating our main result, and postpone a discussion of the background and history to Section~\ref{sec:history}. The \emph{Duarte model}\footnote{See Section~\ref{sec:modifiedDuarte} for a discussion of the closely-related \emph{modified Duarte model}.} is defined as follows. Let
\[
\D := \Big\{ \big\{ (-1,0), (0,1) \big\}, \big\{ (-1,0), (0,-1) \big\}, \big\{ (0,1), (0,-1) \big\} \Big\},
\]
denote the collection of 2-element subsets of $\big\{ (-1,0), (0,1), (0,-1) \big\}$, and let $\Z_n^2$ denote the two-dimensional discrete torus. Given a set $A\subset \Z_n^2$ of initially \emph{infected} sites, set $A_0 = A$, and define for each $t \geq 0$, 
\[
A_{t+1} := A_t \cup \big\{ x \in \Z_n^2 \,:\, x + X \subset A_t \text{ for some } X \in \D \big\}.
\]
Thus, a site $x$ becomes infected at time $t+1$ if the translate by $x$ of one of the sets of $\D$ is already entirely infected at time $t$, and infected sites remain infected forever. The set of eventually infected sites is called the \emph{closure} of $A$, and is denoted by $[A]_\D := \bigcup_{t \geq 0} A_t$. We say that $A$ \emph{percolates} if $[A]_\D = \Z_n^2$.

Let us say that a set $A \subset \Z_n^2$ is \emph{$p$-random} if each of the sites of $\Z_n^2$ is included in $A$ independently with probability $p$, and denote the corresponding probability measure by $\P_p$. The  \emph{critical probability} is defined to be
\begin{equation}\label{def:pc}
p_c(\Z_n^2,\D) := \inf \Big\{ p \in [0,1] \,:\, \P_p\big( [A]_\D = \Z_n^2 \big) \geq 1/2 \Big\};
\end{equation}
that is, the value of $p$ at which percolation becomes likely.

Schonmann~\cite{Sch3} proved that the critical probability of the Duarte model on the plane $\Z^2$ is 0, and Mountford~\cite{Mount} determined $p_c(\Z_n^2,\D)$ up to a constant factor. Here we determine the following sharp threshold.

\begin{theorem}\label{thm:Duarte}
$$p_c\big( \Z_n^2,\D \big) = \bigg( \frac{1}{8} + o(1) \bigg) \frac{(\log \log n)^2}{\log n}$$
as $n \to \infty$.
\end{theorem}

The constant $1/8$ in the theorem arises from the typical growth of a `droplet' in the following way. A droplet of height $(c/p) \log(1/p)$ has width about $p^{-1-c}$, which implies that the `cost' of each vertical step of size 2 is roughly $p^{1-c}$. Integrating the logarithm of this function, we obtain $\int_0^1 \frac{1 - c}{2}\, \textup{d}c = 1/4$. The final factor of 2 is due to there being roughly $n^2$ droplets in $\Z_n^2$. For more details, see Section~\ref{sec:upper}.

Our proof of Theorem~\ref{thm:Duarte} relies heavily on the techniques introduced in~\cite{BDMS}, where we proved a weaker result in much greater generality (see Theorem~\ref{thm:BDMS}, below). The key innovation of this paper is the use of non-polygonal `droplets' (see Section~\ref{sec:tools}), which seem to be necessary for the proof, and significantly complicate the analysis. In particular, we will have to work very hard in order to obtain sufficiently strong bounds on the probabilities of suitable `crossing events' (see Section~\ref{sec:crossings}). On the other hand, by encoding the growth using a single variable, these droplets somewhat simplify some other aspects of the proof.  

\subsection{Background and motivation}\label{sec:history}

The study of bootstrap processes on graphs goes back over 35 years to the work of Chalupa, Leath and Reich~\cite{CLR}, and numerous specific models have been considered in the literature. Motivated by applications to statistical physics, for example the Glauber dynamics of the Ising model~\cite{FSS,MGlauber} and kinetically constrained spin models~\cite{CMRT}, the underlying graph is often taken to be $d$-dimensional, and the initial set $A$ is usually chosen randomly. The most extensively-studied of these processes is the classical `$r$-neighbour model' in $d$ dimensions, in which a vertex of $\Z^d$ becomes infected as soon as it acquires at least $r$ already-infected nearest neighbours. The sharp threshold for this model in full generality was obtained by Balogh, Bollob\'as, Duminil-Copin and Morris~\cite{BBDM} in 2012, building on a series of earlier partial results in~\cite{BBM3d,CC,CM,Hol,AL,Sch1}. Their result stated that
\[
p_c\big( \Z_n^d,\n_r^d \big) = \bigg( \frac{\lambda(d,r) + o(1)}{\log_{(r-1)} n} \bigg)^{d-r+1}
\]
as $n \to \infty$, for some explicit constant $\lambda(d,r)$, where the left-hand side is defined as in~\eqref{def:pc}, except replacing $\D$ by $\n_r^d$, the collection of the $\binom{2d}{r}$ $r$-element subsets of the neighbourhood of $\0$ in $\Z^d$. The special case $d=r=2$, a result analogous to Theorem~\ref{thm:Duarte} for the 2-neighbour model in $\Z^2$, was obtained by Holroyd~\cite{Hol} in 2003, who showed that in fact $\lambda(2,2) = \pi^2/18$.

More recently, a much more general family of bootstrap-type processes was introduced and studied by Bollob\'as, Smith and Uzzell~\cite{BSU}. To define this family in two dimensions, let $\U = \{X_1,\ldots,X_m\}$ be a finite collection of finite subsets of $\Z^2 \setminus \{\textbf{0}\}$, and replace $\D$ by $\U$ in each of the definitions above. The key discovery of~\cite{BSU} was that the family of such monotone cellular automata can be partitioned into three classes, each with completely different behaviour. Roughly speaking, one of the following holds:
\begin{itemize}
\item $\U$ is `supercritical' and has polynomial critical probability.\vspace{0.1cm}
\item $\U$ is `critical' and has poly-logarithmic critical probability.\vspace{0.1cm}
\item $\U$ is `subcritical' and has critical probability bounded away from zero. %\vspace{0.1cm} 
\end{itemize}
We remark that the first two statements were proved in~\cite{BSU}, and the third by Balister, Bollob\'as, Przykucki and Smith~\cite{BBPS}. Note that both the Duarte model and the 2-neighbour model are critical, while the 1-neighbour model is supercritical and the 3-neighbour model is subcritical.\footnote{There also exist many non-trivial examples of supercritical and subcritical models.}

For critical models, much more precise bounds were obtained recently by the authors of this paper~\cite{BDMS}. Since this paper should be seen as a direct descendent of that work, we will spend a little time developing the definitions necessary for the statement of the main theorem of~\cite{BDMS}. 

\begin{definition}
For each $u \in S^1$, let $\H_u := \{x \in \Z^2 : \< x,u \> < 0 \}$ denote the discrete half-plane whose boundary is perpendicular to $u$. Given $\U$, define
$$\stab = \stab(\U) = \big\{ u \in S^1 : [\H_u]_\U = \H_u \big\}.$$
The model $\U$ is said to be \emph{critical} if there exists a semicircle in $S^1$ that has finite intersection with $\stab$, and if every open semicircle in $S^1$ has non-empty intersection with $\stab$.
\end{definition}
 
We call the elements of $\stab$ \emph{stable directions}. Note that for the Duarte model 
$$\stab(\D) = \big\{ u \in S^1 : \theta(u) \in \{0\} \cup [\pi/2,3\pi/2] \big\},$$ 
where $\theta(u)$ is the (canonical) angle of $u$ in radians. Thus the open semicircle $(-\pi/2,\pi/2)$ contains exactly one stable direction, and every other open semicircle contains an infinite number of stable directions. The next definition allows us to distinguish between different types of stable direction.

\begin{definition}\label{de:alpha}
Given a rational direction $u \in S^1$, the \emph{difficulty} of $u$ is
\[
\alpha(u) :=
\begin{cases}
\min\big\{ \alpha^+(u), \alpha^-(u) \big\} &\text{if } \alpha^+(u) < \infty \text{ and } \alpha^-(u)<\infty \\
\hfill \infty \hfill & \text{otherwise,}
\end{cases}
\]
where $\alpha^+(u)$ (respectively $\alpha^-(u)$) is defined to be the minimum (possibly infinite) cardinality of a set $Z \subset \Z^2$ such that $[\H_u \cup Z]_\U$ contains infinitely many sites of the line $\ell_u := \{x \in \Z^2 : \< x,u \> = 0 \}$ to the right (resp. left) of the origin.
\end{definition}

Writing $u^+$ for the isolated element of $\S(\D)$ (so $\theta(u^+) = 0$), we have $\alpha(u^+) = 1$ and $\alpha(u) = \infty$ for every $u \in \S(\D) \setminus \{u^+\}$. More precisely, writing $u^*$ for the element of $S^1$ with $\theta(u^*) = \pi/2$, we have 
$$\alpha^+(u^*) = \alpha^-(-u^*) = 1 \qquad \textup{ and} \qquad \alpha^-(u^*) = \alpha^+(-u^*) = \infty,$$
and $\alpha^+(u) = \alpha^-(u) = \infty$ for every $u \in \S(\D) \setminus \{u^+,u^*,-u^*\}$. Writing $\C$ for the collection of open semicircles of $S^1$, we define the \emph{difficulty of $\U$} to be
\begin{equation}\label{eq:alphaU}
\alpha = \alpha(\U) := \min_{C \in \C} \, \max_{u \in C} \, \alpha(u),
\end{equation}
so $\alpha(\D) = 1$. The final definition we need is as follows.

\begin{definition}\label{de:balanced}
A critical update family $\U$ is \emph{balanced} if there exists a closed semicircle $C$ such that $\alpha(u) \leq \alpha$ for all $u\in C$. It is said to be \emph{unbalanced} otherwise.
\end{definition}

Note that $\D$ is unbalanced, since every closed semicircle in $S^1$ contains a point of infinite difficulty. The main theorem of~\cite{BDMS} was as follows.

\begin{theorem}\label{thm:BDMS}
Let $\U$ be a critical two-dimensional bootstrap percolation update family and let $\alpha=\alpha(\U)$.
\begin{enumerate}
\item If $\U$ is balanced, then
\[
p_c\big( \Z_n^2,\U \big) = \Theta \bigg( \frac{1}{\log n} \bigg)^{1/\alpha}.
\]
\item If $\U$ is unbalanced, then
\[
p_c\big( \Z_n^2,\U \big) = \Theta \bigg( \frac{(\log \log n)^2}{\log n} \bigg)^{1/\alpha}.
\]
\end{enumerate}
\end{theorem}

By the remarks above, it follows from Theorem~\ref{thm:BDMS} that $p_c\big( \Z_n^2,\D \big) = \Theta\big( \frac{(\log \log n)^2}{\log n} \big)$, as was first proved by Mountford~\cite{Mount}. Sharp thresholds  (that is, upper and lower bounds which differ by a factor of $1 + o(1)$) are known in some special cases. For example, Duminil-Copin and Holroyd~\cite{DH} obtained such a result for symmetric, balanced, threshold models (that is, balanced models in which $\U$ consists of the $r$-subsets of some centrally symmetric set), and Duminil-Copin and van Enter~\cite{DE} determined the sharp threshold for the unbalanced model with update rule $\A$ consisting of the 3-subsets of $\big\{ (-2,0), (-1,0), (0,1), (0,-1), (1,0), (2,0) \big\}$, proving that
$$p_c\big( \Z_n^2,\A \big) = \bigg( \frac{1}{12} + o(1) \bigg) \frac{(\log \log n)^2}{\log n}$$
as $n \to \infty$. This was, until now, the only sharp threshold result known for an unbalanced critical bootstrap process in two dimensions.

The key property which makes the process with update rule $\A$ easier to deal with than the Duarte model is symmetry, in particular the fact that $\alpha^+(u^*) = \alpha^-(u^*) = 2$. As a result of this symmetry, the droplets are rectangles, and there is a natural way to partition vertical growth into steps of size one. The Duarte model also exhibits symmetry, but of a weaker kind: there exists a set of four pairwise-opposite stable directions. Theorem~\ref{thm:Duarte} is the first result of its kind for a model (balanced or unbalanced) that only exhibits this weaker notion of symmetry. 

The proof of Theorem~\ref{thm:Duarte} follows in outline that of Theorem~\ref{thm:BDMS} in the case of unbalanced `drift' models (that is, models for which $\alpha(u^*) = \alpha(-u^*) = \infty$), with a few important differences. In particular, we will use the `method of iterated hierarchies' (see Section~\ref{sec:tools}), but the droplets we use to control the growth will \emph{not} be polygons. Instead, they will grow upwards as they grow rightwards; crucially, however, in a \emph{deterministic} fashion. This means that their size will depend on only one parameter (their height), rather than two, as in the case of a rectangle. As noted above, this has the pleasantly surprising consequence of simplifying some of the analysis, although (rather less surprisingly) its overall effect is to significantly increase the number of technical difficulties that will need to be overcome, as we shall see in Sections~\ref{sec:tools} and~\ref{sec:crossings}. This is the first time that non-polygonal droplets have been used in bootstrap percolation, and we consider this innovation to be the key contribution of this paper.  

The rest of this paper is organised as follows. We begin in the next section by giving the (relatively easy) proof of the upper bound in Theorem~\ref{thm:Duarte}. Then, in Section~\ref{sec:tools}, we prepare for the proof of the lower bound by defining precisely the droplet described above, by stating a number of other key definitions, and by recalling some fundamental definitions from~\cite{BDMS} and~\cite{Hol}. The most technical part of the paper is Section~\ref{sec:crossings}, in which we prove precise bounds on the probability that a droplet grows to `span' a slightly larger droplet. In Section~\ref{sec:small} we use the `method of iterated hierarchies' to bound the probability that relatively small droplets are internally spanned, and in Section~\ref{sec:large} we deduce the corresponding bound for large droplets, and hence complete the proof of Theorem~\ref{thm:Duarte}. Finally, in Section~\ref{sec:open}, we discuss possible extensions of our techniques to more general two-dimensional processes, and the (much harder) problem of extending these methods to higher dimensions.

\section{The upper bound}\label{sec:upper}

The upper bound in Theorem~\ref{thm:Duarte} is relatively straightforward. We will prove the following proposition, which easily implies it (the deduction is given at the end of the section). Given a rectangle $R$ with sides parallel to the axes, let $\edge(R)$ denote its right-hand side. 

\begin{prop}\label{prop:upper}
For every $\eps > 0$, there exists $p_0(\eps) > 0$ such that the following holds. Let $0 < p \le p_0(\eps)$, set $a = 1 / p^{5}$ and $b = 1 / p^{3}$, and let $R$ be a rectangle of width $a$ and height $b$. Then 
\[
\P_p\Big( \edge(R) \subset [R \cap A] \Big) \ge \exp\Bigg( - \frac{1+\eps}{4p} \left( \log \frac{1}{p}  \right)^2 \Bigg).
\]
\end{prop}

The growth structure we use to prove Proposition~\ref{prop:upper} is illustrated in Figure~\ref{fig:upper}. We will define rectangles $R_0,\ldots,R_k$, where $k:=1/\eps$, and bound the probability that $R_0\subset [R_0\cap A]$, % internally filled not defined
and that $R_0$ then grows to infect the other rectangles in turn. (Note that if $1/\eps$ is not an integer then we may replace $\eps$ by $1/\lceil 1/\eps\rceil$.)

\begin{figure}[ht]
  \centering
  \begin{tikzpicture}[>=latex]
    \draw (-0.2,0) rectangle (0,1) (0,0) rectangle (2,1) (2,0) rectangle (5.5,2) (5.5,0) rectangle (12,3);
    \draw [densely dashed] (0,1) -- (0,2) -- (2,2) -- (2,3) -- (5.5,3) -- (5.5,4) -- (12,4);
    \draw (12,0) -- (12.2,0) (12,3) -- (12,4) -- (12.2,4);
    \draw [densely dashed] (12.3,0) -- (12.5,0) (12.3,4) -- (12.5,4);
    \node at (1,0.5) {$R_1$};
    \node at (3.75,1) {$R_2$};
    \node at (8.75,1.5) {$R_3$};
%    \draw (-0.2,0.5) -- (-0.5,0.5) node [left] {$R_0$};
    \node at (-0.6,0.5) {$R_0$};
    \draw [<->] (2.3,2) -- node [right] {$h$} (2.3,3);
    \draw [<->] (0,-0.3) -- node [below] {$w_1$} (2,-0.3);
    \draw [<->] (2,-0.3) -- node [below] {$w_2$} (5.5,-0.3);
    \draw [<->] (5.5,-0.3) -- node [below] {$w_3$} (12,-0.3);
  \end{tikzpicture}
  \caption{Our proof of the upper bound of Theorem~\ref{thm:Duarte} shows that one (asymptotically) optimal route to percolation of $\Z_n^2$ is, somewhere in the torus, for infection to spread in the manner depicted in the figure. From $R_0$ infection spreads rightwards through $R_1$, then upwards from $R_1$ to $R_1'$ (which is the union of $R_1$ and the dashed region above), then rightwards again into $R_2$, and so on.}\label{fig:upper}
\end{figure}
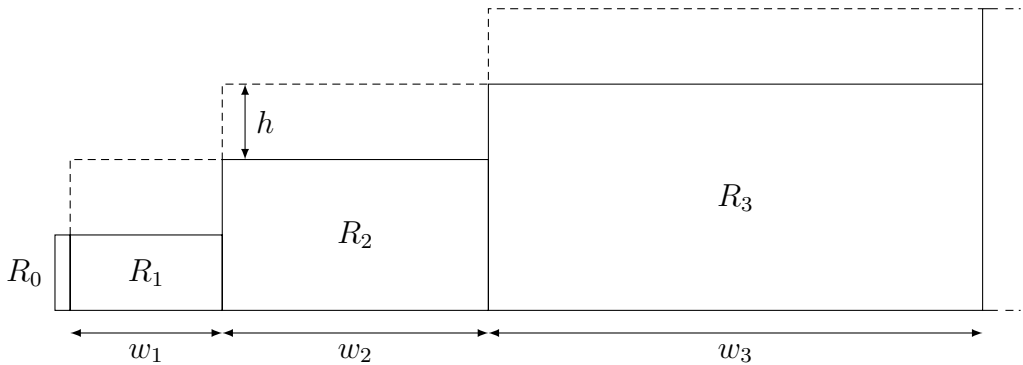

Let us denote the discrete rectangle with opposite corners $(a,b)$ and $(c,d)$ by
\[
R\big((a,b),(c,d)\big) := \big\{ (x,y) \in \Z^2 \,:\, a \leq x \leq c \text{ and } b \leq y \leq d \big\}.
\] 
Assume that $\eps > 0$ and $0 < p < p_0(\eps)$ are both sufficiently small, and set
\[
h := \frac{\eps}{p}\log\frac{1}{p} \qquad \text{and} \qquad w_i := p^{-1-i\eps}
\]
for each $i \in [k]$. We define
\[
R_0 := R_0' := R\Big((0,0),\big(0,h\big)\Big)
\]
and, for each $i \in [k]$,
\[
R_i := R\Bigg( \bigg(1+\sum_{j=1}^{i-1} w_j \, , \, 0 \bigg), \, \bigg(\sum_{j=1}^i w_j \, , \, ih \bigg) \Bigg)
\]
and
\[
R_i' := R\Bigg( \bigg(1+\sum_{j=1}^{i-1} w_j \, , \, 0 \bigg), \, \bigg(\sum_{j=1}^i w_j \, , \, (i+1)h \bigg) \Bigg),
\]
Thus the $R_i$ are rectangles whose heights grow linearly and widths exponentially in $i$, and consecutive rectangles are adjacent. The rectangle $R_i'$ contains $R_i$ and has height equal to that of $R_{i+1}$. The set-up is depicted in Figure~\ref{fig:upper}.

We first prove the following easy lemma. % in the proof of Proposition~\ref{prop:upper}.

\begin{lemma}\label{lem:upper:basicstep}
For each $i \in [k]$,
$$\P_p\big( \edge(R_i') \subset [ \edge(R'_{i-1}) \cup (R_i' \cap A) ] \big) \ge e^{-2/p} \cdot p^{(1-i\eps+\eps^2)h/2}.$$

\end{lemma}

\begin{proof}
Note first that, since a single infected site in each column is sufficient for horizontal growth, we have
%\begin{equation}\label{eq:upper:Ri}
$$\P_p\big( R_i \subset [ \edge(R'_{i-1}) \cup (R_i \cap A) ] \big) \geq \big( 1 - (1-p)^{ih} \big)^{w_i} \geq \big( 1-p^{i\eps} \big)^{w_i} \geq e^{-2/p},$$
%\end{equation}
since $p^\eps$ is sufficiently small. Now suppose that $R_i$ is already completely infected, and observe that a single element of $A$ in the row two above $R_i$ causes all elements to its right in these two rows to become infected (see Figure~\ref{fig:upper2}). Note that the probability of finding at least one site of $A$ in a collection of $w_i/h$ sites is
\[
1 - (1-p)^{w_i/h} \ge 1 - \exp\bigg( - \frac{p^{1-i\eps}}{\eps \log(1/p)} \bigg) \ge p^{1-i\eps+\eps^2},
\]
since $\eps p^{\eps^2} \log(1/p) < 1/2$. It follows that  
%\begin{equation}\label{eq:upper:Ri'}
$$\P_p\big( \edge(R_i') \subset [ R_i \cup (R_i' \cap A) ] \big) \geq p^{(1-i\eps+\eps^2)h/2},$$
%\end{equation}
as required.
\end{proof}

\begin{figure}[ht]
  \centering
  \begin{tikzpicture}[>=latex,scale=1.3]
    \draw (0,0) -- (0,1) -- (8.3,1) (8,0) -- (8,3.4) -- (8.3,3.4) -- (8.3,0);
    \draw [densely dashed] (0,1.3) -- (8,1.3) (0,1.6) -- (8,1.6);
    \draw [densely dashed] (2,1.9) -- (8,1.9) (2,2.2) -- (8,2.2);
    \draw [densely dashed] (4,2.5) -- (8,2.5) (4,2.8) -- (8,2.8);
    \draw [densely dashed] (6,3.1) -- (8,3.1) (6,3.4) -- (8,3.4);
    \draw [densely dashed] (0,1) -- (0,1.6) (2,1.6) -- (2,2.2) (4,2.2) -- (4,2.8) (6,2.8) -- (6,3.4);
    \draw (1.2,1.3) rectangle (1.5,1.6);
    \draw (2.5,1.9) rectangle (2.8,2.2);
    \draw (4.9,2.5) rectangle (5.2,2.8);
    \draw (7.4,3.1) rectangle (7.7,3.4);
    \node [draw,cross out,inner sep=0pt,minimum size=0.15cm] at (1.35,1.45) {};
    \node [draw,cross out,inner sep=0pt,minimum size=0.15cm] at (2.65,2.05) {};
    \node [draw,cross out,inner sep=0pt,minimum size=0.15cm] at (5.05,2.65) {};
    \node [draw,cross out,inner sep=0pt,minimum size=0.15cm] at (7.55,3.25) {};
%    \draw (1.2,1) -- (1.2,1.6) -- (8,1.6);
    \node at (4,0.5) {$R_i$};
    \draw (8.15,1.7) -- ++(0.5,0) node [right] {$\edge(R_i')$};
    \draw [<->] (0,2.5) -- node [above] {$w_i/h$} (2,2.5);
  \end{tikzpicture}
  \caption{Upwards growth through $R_i'$. With $R_i$ and the four marked sites already infected, the whole of $\edge(R_i')$ becomes infected.}\label{fig:upper2}
\end{figure}
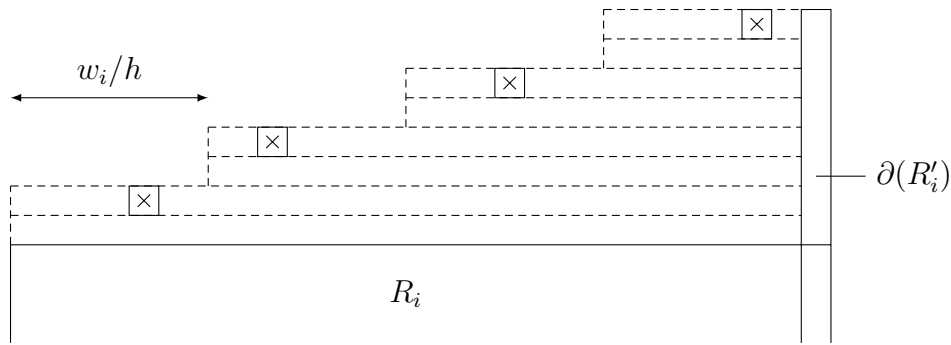

Now set $\hat{w}:=w_1+\dots+w_k$ and
\[
\hat{R}_0 := R\Bigg( (0,0) \, , \, \bigg( \hat{w} \, , \, \frac{1+\eps}{p}\log\frac{1}{p} \bigg) \Bigg). 
\]
The next lemma follows easily from Lemma~\ref{lem:upper:basicstep}. 

\begin{lemma}\label{lem:upper:Rzero}
We have
\[
\P_p\big( \edge(\hat{R}_0) \subset [ \hat{R}_0\cap A ] \big) \ge \exp \Bigg( -\frac{1+2\eps}{4p}\bigg(\log\frac{1}{p}\bigg)^2 \Bigg).
%p^{(1 + \eps^2)(k+1)h/2}.
\]
\end{lemma}

\begin{proof}
Note that $\edge(\hat{R}_0) = \edge(R_k')$, and that 
$$\P_p\big( R_0 \subset [R_0 \cap A] \big) \geq p^{\lfloor h/2 \rfloor + 1},$$
since if every second element of $R_0$ is in $A$ then $R_0 \subset [R_0 \cap A]$. Therefore
\[
\P_p\big( \edge(\hat{R}_0) \subset [ \hat{R}_0\cap A ] \big) \geq p^{\lfloor h/2 \rfloor + 1} \cdot \prod_{i=1}^k \P_p\big( \edge(R_i') \subset [ \edge(R'_{i-1}) \cup (R_i' \cap A) ] \big).
\]
By Lemma~\ref{lem:upper:basicstep}, the right-hand side is at least
$$p^{\lfloor h/2 \rfloor + 1} e^{-2k/p} \prod_{i=1}^k p^{(1-i\eps+\eps^2)h/2} \ge e^{-2k/p} \big( p^{(k+1)h/2} \big)^{1 - \eps k / 2 + \eps^2} \ge p^{(1 + 3\eps^2)h(k+1)/4},$$
since $p$ is sufficiently small and $\eps k = 1$. Recalling that $h = \frac{\eps}{p}\log\frac{1}{p}$, and noting that $(1 + 3\eps^2)(1 + \eps) < 1 + 2\eps$ since $\eps$ is sufficiently small, the claimed bound follows.
\end{proof}

We can now easily complete the proof of Proposition~\ref{prop:upper}. Indeed, once we have infected $\edge(\hat{R}_0)$ it is relatively easy to grow $p^{-2-\eps}$ steps to the right, then $p^{-1-\eps/2}$ steps upwards, then $p^{-5}$ steps right, and finally $p^{-3}$ steps up. For completeness we spell out the details below.

\begin{proof}[Proof of Proposition~\ref{prop:upper}]
Recall that $R = R\big[ (0,0), (p^{-5},p^{-3}) \big]$. We claim that 
\begin{equation}\label{eq:prop:upper:claim}
\P_p\Big( \edge(R) \subset \big[ \edge(\hat{R}_0) \cup (R \cap A) \big] \Big) \ge e^{-O(1/p)}.
\end{equation}
In order to prove~\eqref{eq:prop:upper:claim}, we will need to define three more rectangles. First, set
\[
\hat{R}_1 = R\Bigg( ( \hat{w}+1 , 0 )  \, , \, \bigg( \hat{w} + p^{-2-\eps} \, , \, \frac{1+\eps}{p}\log\frac{1}{p} \bigg) \Bigg),
\]
and observe that 
\[
\P_p\Big( \hat{R}_1 \subset \big[ \edge(\hat{R}_0) \cup (\hat{R}_1 \cap A) \big] \Big) \ge \Big( 1 - (1-p)^{h(\hat{R}_1)} \Big)^{w(\hat{R}_1)} \ge e^{-O(1/p)},
\]
since $\exp\big( - p \cdot h(\hat{R}_1) \big) = p^{-(1+\eps)}$ and $p^{-(1+\eps)} \cdot w(\hat{R}_1) = 1/p$. Next, set  
\[
\hat{R}_1' = R\Big( ( \hat{w}+1 , 0 )  \, , \, \big( \hat{w} + p^{-2-\eps} \, , \, p^{-1-\eps/2} \big) \Big),
\]
and observe that
\[
\P_p\Big( \edge\hat{R}'_1 \subset \big[ \hat{R}_1 \cup (\hat{R}'_1 \cap A) \big] \Big) \ge \Big( 1 - (1-p)^{w(\hat{R}_1) / h(\hat{R}'_1)} \Big)^{h(\hat{R}'_1)/2} > \frac{1}{2}
\]
since $\exp\big( - p \cdot w(\hat{R}_1) / h(\hat{R}'_1) \big) = \exp( - p^{-\eps/2} ) \ll p^2$ and $h(\hat{R}'_1) \ll p^{-2}$. Finally, set
\[
\hat{R}_2 = R\Big( \big( \hat{w} + p^{-2-\eps} + 1 \, , \, 0 \big)  \, , \, \big( p^{-5} \, , \, p^{-1-\eps/2} \big) \Big),
\]
and observe that
\[
\P_p\Big( \hat{R}_2 \subset \big[ \edge(\hat{R}'_1) \cup (\hat{R}_2 \cap A) \big] \Big) \ge \Big( 1 - (1-p)^{h(\hat{R}_2)} \Big)^{w(\hat{R}_2)} > \frac{1}{2}
\]
since $\exp\big( - p \cdot h(\hat{R}_2) \big) \ll p^{-5}$ and $w(\hat{R}_2) \le 1/p^5$, and
\[
\P_p\Big( \edge(R) \subset \big[ \hat{R}_2 \cup (R \cap A) \big] \Big) \ge \Big( 1 - (1-p)^{w(\hat{R}_2) / h(R)} \Big)^{h(R)/2} > \frac{1}{2}
\]
since $\exp\big( - p \cdot w(\hat{R}_2) / h(R) \big) \ll p^3$ and $h(R) = p^{-3}$. This proves~\eqref{eq:prop:upper:claim}, and, together with Lemma~\ref{lem:upper:Rzero}, it follows that
\[
\P_p\big( \edge R \subset [ R\cap A ] \big) \geq \exp \Bigg( -\frac{1+3\eps}{4p}\bigg(\log\frac{1}{p}\bigg)^2 \Bigg).
\]
Since $\eps$ was arbitrary, the proposition follows.
\end{proof}

Finally, let us deduce the upper bound of Theorem~\ref{thm:Duarte} from Proposition~\ref{prop:upper}.

\begin{proof}[Proof of the upper bound of Theorem~\ref{thm:Duarte}]
Fix $\lambda > 1/8$, and set
\begin{equation}\label{eq:pdef}
p = \frac{\lambda (\log \log n)^2}{\log n}.
\end{equation}
We will show that, with high probability as $n \to \infty$, a $p$-random subset $A \subset \Z_n^2$ percolates. Observe first that $\Z_n^2$ contains $\Omega\big( p^8 n^2 \big)$ disjoint translates of the rectangle $R = R\big[ (0,0), (p^{-5},p^{-3}) \big]$. Since
\begin{equation}\label{eq:upperfinal}
\exp\Bigg( - \frac{1+\eps}{4p} \left( \log \frac{1}{p} \right)^2 \Bigg) \ge \exp\bigg( - \frac{1+\eps}{4\lambda} \cdot \log n \bigg) \ge n^{-2+\eps}
\end{equation}
if $\eps > 0$ is sufficiently small, it follows from Proposition~\ref{prop:upper} that, with high probability, there exists such a translate with $\edge(R) \subset [R \cap A]$.

To complete the proof, simply observe that with probability at least
\[
1 - 2 n^2 \big( 1 - p \big)^{1 / p^3} \ge 1 - \frac{1}{n},
\]
there does not exist a (horizontal or vertical) line of $1/p^3$ consecutive sites of $\Z_n^2$ that contains no element of $A$. But if this holds then the set $\edge(R) \cup A$ clearly percolates in $\Z_n^2$, and so we are done.
\end{proof}

\section{Droplets, spanning, and iterated hierarchies}\label{sec:tools}

\subsection{Droplets and the growth of infected regions}\label{sec:droplets}

We are now ready to start the main part of the proof of Theorem~\ref{thm:Duarte}: the proof of the lower bound on the critical probability. We begin by formally introducing the curved droplets we shall use to control the growth of an infection. This will then allow us to state the key result (Proposition~\ref{prop:lower}) we need in the lead up to Theorem~\ref{thm:Duarte}. Later, in Section~\ref{sec:spanning}, we establish certain deterministic facts about `internally spanned droplets' (see Definition~\ref{def:ispan} below), and in Section~\ref{sec:hierarchies} we briefly recall the definitions and properties we shall need for the `method of iterated hierarchies'.

We begin by defining a \emph{droplet}. The definition is quite subtle, and is chosen both to reflect the typical growth of the infected set, and to facilitate our proof of Theorem~\ref{thm:Duarte}. For simplicity, we will work in $\Z^2$ (and $\R^2$) throughout this section, though all of the definitions and lemmas below can be easily extended to the setting of $\Z_n^2$.

\begin{definition}\label{def:droplet}
Given $\eps > 0$ and $p > 0$, a \emph{Duarte region} $D^* \subset \R^2$ is a set of the form
\begin{equation}\label{eq:Dstar}
D^* = (a,b) + \big\{ (x,y) \in \R^2 : 0 \le x \le w, \, |y| \le f(x) \big\},
\end{equation}
for some $a,b,w \in \R$, where $f \colon [0,\infty) \to [0,\infty)$ is the function
\[
f(x) := \frac{1}{2p}\log\left(1+\frac{\eps^3 px}{\log 1/p}\right).
\]
A \emph{Duarte droplet} (or simply, as we shall usually say, a \emph{droplet}) $D \subset \Z^2$ is the intersection of a Duarte region with $\Z^2$. Thus, $D$ is a Duarte droplet if and only if there exists a Duarte region $D^*$ such that $D = D^* \cap \Z^2$.
\end{definition}

Let us make an easy but important observation.

\begin{obs}\label{obs:DofK}
Given a bounded set $U\subset\R^2$, there is a (unique) minimal Duarte region $D^*(U)$ containing $U$.
\end{obs}

If $K\subset\Z^2$ is finite, we define the \emph{minimal droplet containing $K$} to be $D(K) := D^*(K) \cap \Z^2$. Notice that $K \subset D(K)$ and that $D$ is the identity function on droplets.

Observation~\ref{obs:DofK} allows us to make the following definitions. Given a bounded set $U\subset\R^2$ and $a,b,w$ such that the right-hand side of~\eqref{eq:Dstar} is $D^*(U)$, we define the \emph{height} and \emph{width} of $U$ by $h(U) := 2f(w)+1$ and $w(U) := w$, respectively. We call the point $(a,b)$ the \emph{source} of $U$. Letting
\[
c := \sup \{ x \in \R \,:\, (x,y)\in U \text{ for some } y \in \R \},
\]
we write
\[
\edge(U) := \big\{ (c,y) \in U \,:\, y\in\R \big\}.
\]
Informally we think of $\edge(U)$ as being the right-hand side of $U$.\footnote{This generalizes the definition of $\edge(R)$ for a rectangle $R$, given in Section~\ref{sec:upper}.}
%The \emph{height} and \emph{width} of $D$ are defined by $h(D) := 2f(w)+1$ and $w(D) = w$ (not uniquely defined!), respectively. We call the point $(a,b)$ the \emph{source} of $D$.
We can now make another easy but important observation, the proof of which is immediate from the convexity of $f$.

\begin{obs}\label{obs:edges}
If $D_1^*$ and $D_2^*$ are Duarte regions such that $\edge(D_1^*) \subset D_2^*$, then $D_1^* \subset D_2^*$.
\end{obs}

It is worth noting that the reason for defining Duarte regions as well as (Duarte) droplets, and for defining heights and widths of droplets in terms of regions, is that if one were to define everything discretely then certain key lemmas below would be false. For example, it would be more natural to define $D(K)$ to be the smallest droplet containing $K$, but if one were to do that then Lemma~\ref{lem:subadd} would be false. (It would be true with `$+2$' in place of `$+1$', but that would be too weak for the application in Lemma~\ref{lem:extremal}.)

One disadvantage of defining droplets in this way is that it makes the following lemma non-trivial.

\begin{lemma}\label{lem:numdroplets}
There are at most $w^{O(1)}$ droplets $D$ such that the source of $D$ belongs to $(0,1] \times (0,1]$ and the $x$-coordinate of the elements of $\edge(D)$ is equal to $w$.
\end{lemma}

The proof of the lemma is a simple consequence of the following extremal result for set systems. Let us say\footnote{We write $\PP[n]$ for the power set of $[n]$.} that a set $\F \subset \PP[n]$ is a \emph{bi-chain} if it has the following property: for every distinct $A,B \in \F$, there exists $k \in [n]$ such that the following two conditions hold: (a) $A \cap \{1,\dots,k\}$ is a subset of $B \cap \{1,\dots,k\}$, or vice-versa, and (b) $A \cap \{k+1,\dots,n\}$ is a subset of $B \cap \{k+1,\dots,n\}$, or vice-versa.

\begin{lemma}\label{lem:hypercube}
Let $\F \subset \PP[n]$ be a bi-chain. Then $|\F| \leq n^{O(1)}$.
\end{lemma}

\begin{proof}
If $A,B \in \F$ are distinct and have the same cardinality, then without loss of generality we may assume that $A \cap \{1,\dots,k\} \subset B \cap \{1,\dots,k\}$ and $B \cap \{k+1,\dots,n\} \subset A \cap \{k+1,\dots,n\}$. This implies that the sum of the elements of $A$ is strictly greater than the sum of the elements of $B$. So summing over the possible sizes of $|A|$, we have $|\F| \leq n^3$.

Alternatively, one may note that the bi-chain condition implies no set $T \subset [n]$ of size 3 is shattered\footnote{A set $T$ is said to be \emph{shattered} by $\F$ if every subset of $T$ can be obtained as an intersection $A \cap T$, for some $A \in \F$.} by $\F$. To see this, suppose $T = \{i,j,k\}$ is such a set, with $i < j < k$. Then there exist $A,B \in \F$ such that $A \cap T = \{i,k\}$ and $B \cap T = \{j\}$, which contradicts the condition. Hence, by the Sauer--Shelah Theorem~\cite{Sauer,Shelah}, we must have $|\F| \leq O(n^2)$. (Note this is optimal up to the constant factor.)\footnote{The first proof given here is due to Paul Balister and the second is due to Bhargav Narayanan. The authors would like to thank both for bringing these proofs to our attention.} 
\end{proof}

\begin{proof}[Proof of Lemma~\ref{lem:numdroplets}]
Firstly, given a droplet $D$, let $\top(D)$ be the set containing the topmost site of each column of $D$, and similarly define $\bottom(D)$. It is easy to see that a droplet $D$ is uniquely determined by the set $\top(D) \cup \bottom(D)$.

Let $\A$ be the set of droplets $D$ whose source is contained in the unit square $(0,1] \times (0,1]$ and such that the $x$-coordinate of the elements of $\edge(D)$ is $w$. For each $D \in \A$, there are (at most) $2^w$ possibilities for $\top(D)$, since there are only 2 choices for the element of $\top(D)$ at each $x$ coordinate (this is because $f'(x) < 1$ for all $x \geq 0$). Thus there is a natural bijection between the set $\A_\top := \big\{ \top(D) : D \in \A \big\}$ and a subset $\F$ of $\PP[n]$ (the power set of $\{1,\dots,n\}$). Moreover, $\F$ is a bi-chain. This is because any two translations of the curve $\big\{(x,f(x)) : x \geq 0\big\}$ intersect in at most one point. Hence, by Lemma~\ref{lem:hypercube}, we have $|\A_\top| = |\F| \leq w^{O(1)}$. Defining $\A_\bottom$ similarly, it follows that $|\A| \leq |\A_\top| \cdot |\A_\bottom| \leq w^{O(1)}$.
\end{proof}

Let us briefly collect together a few simple facts about $f$, which we shall use repeatedly throughout the paper.

\begin{obs}\label{obs:f}
The function $f$ has the following properties for all $\eps > 0$ and $p > 0$:
\begin{itemize}
\item[$(a)$] $f$ is strictly increasing on $[0,\infty)$.
\item[$(b)$] $f'$ is strictly decreasing (and hence $f$ is convex) on $[0,\infty)$.
\item[$(c)$] $f'(x) = \eps^3(2\log 1/p)^{-1} e^{-2pf(x)}$.
\item[$(d)$] If $f(x) \leq 1/4p$ then
\[
\frac{\eps^3}{4\log 1/p} \leq f'(x) \leq \frac{\eps^3}{2\log 1/p}.
\]
%\item $f(x)\leq 1/px$ for all $x>0$.
\end{itemize}
\end{obs}

Next, let us record a few conventions, also to be used throughout the paper:
\begin{itemize}
\item $\eps>0$ is an arbitrary and sufficiently small constant, and $p>0$ is sufficiently small depending on $\eps$, with $p\to 0$ as $n\to\infty$.
\item Constants implicit in $O(\cdot)$ notation (and its variants) are absolute: they do not depend on $p$, $n$, $\eps$, $k$, or any other parameter.
\item $A$ denotes a $p$-random subset of $\Z_n^2$.
\item $[K] := [K]_\D$ for $K \subset \Z_n^2$ (or $K \subset \Z^2$).
\end{itemize}

The following key definition is based on an idea first introduced in~\cite{BBM3d,BBDM}.\footnote{We emphasize this definition does not correspond to the use of the term `internally spanned' in much of the older literature, where it was used to mean that $[D \cap A]_\D = D$.}

\begin{definition}\label{def:ispan}
A droplet $D$ is said to be \emph{internally spanned} if there exists a set $L \subset [D \cap A]$ that is connected in the graph $\Z^2$, and such that $D=D(L)$. We write $\ispan(D)$ for the event that $D$ is internally spanned.
\end{definition}

%In Lemma~\ref{lem:DS:def} of Section~\ref{sec:spanning} we shall show that the smallest droplet containing a set $S$ is well-defined.

We can now state the key intermediate result in the proof of Theorem~\ref{thm:Duarte}.

\begin{prop}\label{prop:lower}
For every $\eps > 0$, there exists $p_0(\eps) > 0$ such that the following holds. If  $0 < p \le p_0(\eps)$, and $D$ is a droplet with
\[
h(D) \leq \frac{1-\eps}{p} \log \frac{1}{p},
\]
then
\begin{equation}\label{eq:lower}
\P_p\big( \ispan(D) \big) \leq p^{(1 - \eps) h(D) / 4}.
\end{equation}
\end{prop}

In order to deduce the theorem from this result, we will show (see Lemma~\ref{lem:critical:droplet}) that if $A$ percolates then there exists a pair $(D_1,D_2)$ of disjointly internally spanned droplets, satisfying
\[
\max\big\{ h(D_1), h(D_2) \big\} \leq \frac{1-\eps}{p} \log \frac{1}{p} \quad \text{and} \quad h(D_1) + h(D_2) \geq \frac{1-\eps}{p} \log \frac{1}{p} - 1,
\]
with $d(D_1,D_2) \le 2$. The theorem then follows from Proposition~\ref{prop:lower} by using the van den Berg--Kesten inequality and taking the union bound over all such pairs. %We will actually prove a stronger bound for smaller droplets, see Proposition~\ref{prop:largedroplets}.

Our proof of Proposition~\ref{prop:lower} uses the framework of `hierarchies' (see Section~\ref{sec:hierarchies}), which have become a standard tool in the study of bootstrap percolation since their introduction by Holroyd~\cite{Hol} (see e.g.~\cite{BBDM,DE,GHM}). However, in order to limit the number of possible hierarchies (which is needed, since we will use the union bound), the `seeds' of our hierarchies must have size roughly $1/p$. This is a problem, because (unlike in the 2-neighbour setting) there is no easy way to prove a sufficiently strong bound on the probability that such a seed is internally spanned.\footnote{This is, roughly speaking, because a droplet of this height is too long.} Moreover, we shall need a similar bound in order to control the probability of vertical growth, due to the (potential) existence of `saver' droplets (see Definition~\ref{def:partition}). 

We resolve this problem by using the `method of iterated hierarchies'. This technique, which was introduced by the authors in~\cite{BDMS}, allows one to prove upper bounds on the probability that a droplet is internally spanned by induction on its height. It is specifically designed to overcome the issue of there being too many droplets for the union bound to work. The inductive step itself is proved using hierarchies.

Our induction hypothesis is as follows.

\begin{definition}\label{def:ih}
For each $k \ge 0$, let $\ih(k)$ denote the following statement:
\begin{equation}\label{eq:ihbound}
\P_p\big( \ispan(D) \big) \leq p^{(1 - \eps_k) h(D)/2}
\end{equation}
for every droplet $D$ with $h(D) \leq p^{-(2/3)^k} (\log 1/p)^{-1}$, where
\begin{equation}\label{eq:epsk}
\eps_k = \eps^2 \cdot (3/4)^k.
\end{equation}
\end{definition}

It is no accident that the factor of $1/4$ in the exponent in~\eqref{eq:lower} has become a factor of $1/2$ in~\eqref{eq:ihbound}: this has to do with the transition, as a droplet reaches height $1/p$, to it being likely that the droplet grows one more step to the right (see Proposition~\ref{prop:largedroplets}, and also compare with Lemma~\ref{lem:upper:Rzero}).

The statement we need for the proof of Proposition~\ref{prop:lower} is $\ih(0)$; we will prove that this holds in two steps. First, we will prove that $\ih(k)$ holds for all sufficiently large~$k$ (see Lemma~\ref{lem:basecase}); then we will show that $\ih(k) \Rightarrow \ih(k-1)$ for every $k \ge 1$ (see Lemma~\ref{lem:indstep}). The first step will follow relatively easily from the fact (see Lemma~\ref{lem:extremal}) that if $D$ is internally spanned, then $|D \cap A| \ge h(D)/2$. To prove the second step, we will apply the method of hierarchies, using the induction hypothesis to bound the probability that smaller droplets are internally spanned.

\subsection{Spanning and extremal properties of droplets}\label{sec:spanning}

In this section we will recall from~\cite{BDMS} the `spanning algorithm', and deduce some of its key consequences. In particular we will prove that critical droplets exist and, in the next section, we will show that they have `good and satisfied' hierarchies. In order to get started,
we need a way of saying that two sets of sites are sufficiently close to interact in the Duarte model.

% we need to make the following observation. %prove the following lemma.
% The followin lemma is false! It's not unique, it just has a unique width.
% \begin{lemma}\label{lem:DS:def}
% Let $S \subset \Z^2$ be a finite set. Then there exists a unique smallest droplet containing $S$.
% \end{lemma}

%\begin{proof}
%Let $\ell$ be minimal such that there exists a translate of $S$, say $x+S$, entirely contained in the droplet $D$ with source $\0$ and width $\ell$. Then $-x+D$ contains $S$ and is the unique such droplet of minimal width.
%Let $(a,b)$ be the right-most point such that there exists a droplet containing $S$ whose source is $(a,b)$. 
%Simply `push' $S$ into the cone defined by $\pm f$, and draw a vertical line through the right-most point of $S$. %[Is this enough of a proof?]
%Suppose $D_1$ and $D_2$ are both minimal droplets that contain $S$, but neither is contained in the other. First note that, by minimality, there is a point on each of the three sides of $D_1$ and $D_2$ (where the endpoints are included in both sides). 
%\end{proof}

%From now on, we shall write $D(S)$ for the smallest droplet containing $S$.
%Next we need a way of saying that two sets of sites are sufficiently close to interact in the Duarte model.

\begin{definition}\label{def:strongconn}
Define a graph $G_{\text{strong}}$ with vertex set $\Z^2$ and edge set $E$, where $\big\{(a_1,b_1),(a_2,b_2)\big\}\in E$ if and only if
% by saying that sites  are \emph{strongly connected}, written $x\sim y$, if
\[
|a_1-a_2| \leq 1 \qquad \text{and} \qquad |a_1-a_2| + |b_1-b_2| \leq 2.
\]
We say that a set of vertices $K\subset \Z^2$ is \emph{strongly connected} if the subgraph of $G_{\text{strong}}$ induced by $K$ is connected.
\end{definition}

We are ready to recall the spanning algorithm of~\cite[Section~6]{BDMS}, modified in accordance with Definitions~\ref{def:droplet} and~\ref{def:strongconn}.

\begin{spanning}\label{def:spanalg}
Let $K = \{x_1,\ldots,x_{k_0}\}$ be a finite set of sites. Set $\K^0 := \{K_1^0,\ldots,K_{k_0}^0 \}$, where $K_j^0 := \{x_j\}$ for each $1\leq j \leq k_0$. Set $t := 0$, and repeat the following steps until STOP: 
\begin{itemize}
\item[1.] If there are two sets $K_i^t,K_j^t \in \K^t$ such that the set
\begin{equation}\label{eq:mergespanningsets}
\big[ K_i^t \cup K_j^t \big]
\end{equation}
%is connected in the graph $\Z^2$, then set
is strongly connected, then set
\[
\K^{t+1} := \big( \K^t \setminus \{K_i^t,K_j^t \} \big) \cup \big\{ K_i^t \cup K_j^t \big\},
\]
and set $t := t + 1$. 
\item[2.] Otherwise set $T := t$ and STOP.
\end{itemize}
The output of the algorithm is the \emph{span of $K$},
\[
\< K \> := \big\{ D\big( [K_1^T] \big),\ldots,D\big( [K_k^T] \big) \big\},
\]
where $k = k_0 - T$. Finally, we say that a droplet $D$ is \emph{spanned} by a set $K$ if there exists $K'\subset K$ such that $D\in\<K'\>$.
\end{spanning}

We will need a few more-or-less standard consequences of the algorithm above. We begin with a basic but key lemma (cf.~\cite[Lemma~6.8]{BDMS}).

\begin{lemma}\label{lem:span}
A droplet $D$ is internally spanned if and only if $D \in \< D \cap A \>$.
\end{lemma}

\begin{proof}
For every finite set $K$, we have 
%\begin{equation}\label{eq:spanofK}
\[
\< K \> = \big\{D(K_1),\ldots,D(K_k)\big\},
\]
%\end{equation}
where $K_1,\ldots,K_k$ are the strongly connected components of $[K]$. Applying this to $K = D \cap A$, we see that $D \in \< D \cap A \>$ if and only if $D(L) = D$ for some strongly connected component $L$ of $[D \cap A]$. But $[D \cap A] \subset D$, and so this is equivalent to the event that $D$ is internally spanned, since a subset of $\Z^2$ is strongly connected and closed if and only if it is connected in the graph $\Z^2$ and closed.
\end{proof}

The second lemma is an approximate sub-additivity property for strongly connected droplets. This lemma, and the extremal lemma which follows (Lemma~\ref{lem:extremal}), are the main reasons for defining Duarte regions, and for defining the width and height of a droplet in the `continuous' way via Duarte regions. % Indeed, once one has passed from droplets to Duarte regions, Lemma~\ref{lem:subadd} is essentially a triviality, the only subtlety being the the interaction between the Duarte regions and the necessarily discrete definition of strongly connected.

\begin{lemma}\label{lem:subadd}
%Let $K_1$ and $K_2$ be finite subsets of $\Z^2$, and suppose that $[K_1]$, $[K_2]$ and $[K_1\cup K_2]$ are strongly connected. Then
%\[
%h\big( D([K_1\cup K_2]) \big) \leq h\big( D([K_1]) \big) + h\big( D([K_2]) \big) + 1.
%\]
Let $D_1$ and $D_2$ be droplets such that $D_1\cup D_2$ is strongly connected. Then
\[
h\big( D(D_1\cup D_2) \big) \leq h(D_1) + h(D_2) + 1.
\]
\end{lemma}

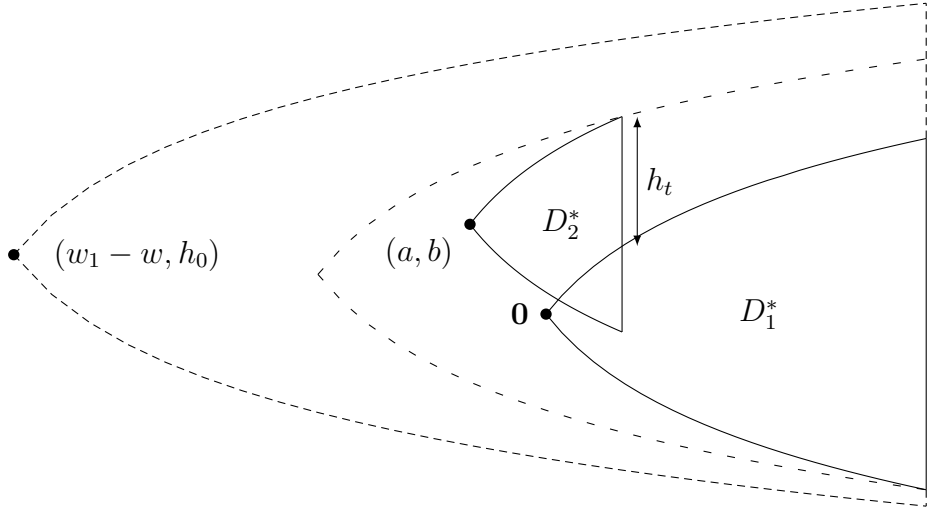
\begin{figure}
  \centering
  \begin{tikzpicture}[>=latex]
    % D^*(union)
    \draw [domain=0:8,loosely dashed] plot (\x,{1.3*ln(1+\x)});
    \draw [domain=0:8,loosely dashed] plot (\x,{-1.3*ln(1+\x)});
%    \draw [domain=0:1,variable=\t,densely dashed] plot (8,{1.3*ln(4)+1.3*ln(9/4)*\t});
    % D^*
%    \draw [domain=0:12,densely dashed] plot (\x-4,{1.3*ln(13/9)+1.3*ln(1+\x)});
%    \draw [domain=0:12,densely dashed] plot (\x-4,{1.3*ln(13/9)-1.3*ln(1+\x)});
    \draw [domain=0:12,densely dashed] plot (\x-4,{1.3*ln(11/9)+1.3*ln(1+\x)});
    \draw [domain=0:12,densely dashed] plot (\x-4,{1.3*ln(11/9)-1.3*ln(1+\x)});
    \draw [domain=0:1,variable=\t,densely dashed] plot (8,{1.3*ln(4)+1.3*ln((13*11)/(9*4))*\t});
    \draw [domain=0:1,variable=\t,densely dashed] plot (8,{-1.3*ln(9)-1.3*ln((13*9)/(9*11))*\t});
    % D_1^*
    \draw [domain=0:5,variable=\t] plot ({\t+3},{1.3*ln((\t+1)/1.5)});
    \draw [domain=0:5,variable=\t] plot ({\t+3},{-1.3*ln((\t+1)*1.5)});
    \draw [domain=-1:1,variable=\t] plot (8,{1.3*ln(6)*\t-1.3*ln(1.5)});
    % D_2^*
    \draw [domain=0:2,variable=\t] plot ({\t+2},{1.3*ln((\t+1)*5/3)});
    \draw [domain=0:2,variable=\t] plot ({\t+2},{-1.3*ln((\t+1)*3/5)});
    \draw [domain=-1:1,variable=\t] plot (4,{1.3*ln(5/3)+1.3*ln(3)*\t});
    % arrows
    \draw [<->,domain=0:1,variable=\t] plot (4.2,{1.3*ln(4/3)+1.3*ln(15/4)*\t});
%    \draw [<->,domain=0:1,variable=\t] plot (8.2,{1.3*ln(4)+1.3*ln(9/4)*\t});
    % labels
    \pgfmathparse{-1.3*ln(1.5)};
    \node at (5.8,\pgfmathresult) {$D_1^*$};
    \pgfmathparse{1.3*ln(5/3)};
    \node at (3.2,\pgfmathresult) {$D_2^*$};
    \pgfmathparse{1.3*ln(4/3)+0.65*ln(15/4)};
    \node at (4.5,\pgfmathresult) {$h_t$};
%    \pgfmathparse{1.3*ln(4)+0.65*ln(9/4)};
%    \node at (10,\pgfmathresult) {$h_0+f(w)-f(w_1)$};
    % new
    \node [fill,circle,inner sep=1.5pt,label=180:$\0$] at (3,{-1.3*ln(1.5)}) {};
    \node [fill,circle,inner sep=1.5pt,label={-160:$(a,b)$}] at (2,{1.3*ln(5/3)}) {};
    \node [fill,circle,inner sep=1.5pt,label={[label distance=0.3cm]0:$(w_1-w,h_0)$}] at (-4,{1.3*ln(11/9)}) {};
  
  \end{tikzpicture}
  \caption{The Duarte regions from the proof of Lemma~\ref{lem:subadd}. The inner dashed region is $D^*(D_1^* \cup D_2^*)$ and the outer dashed region is $D^*$. In this example $h_b=0$. Note that $D^*$ is much larger than $D^*(D_1^* \cup D_2^*)$. Since $D^*$ is defined so that $h(D^*) = h(D_1^*) + h(D_2^*) + 1$ by~\eqref{eq:defw}, this discrepancy occurs whenever there is a large overlap between the droplets.}
  \label{fig:subadd}
\end{figure}

\begin{proof}
It will be convenient to pass to the continuous setting, so let $D_i^*:=D^*(D_i)$ for $i=1,2$. We shall prove that
\[
h(D^*) \leq h(D_1^*) + h(D_2^*) + 1
\]
for some Duarte region $D^*$ containing $D_1^*\cup D_2^*$. Since $h(D)$ is defined to be $h\big(D^*(D)\big)$ for any droplet $D$, and since $D^* \supset D^*(D_1^*\cup D_2^*) \supset D^*(D_1\cup D_2)$, this would imply the result.

We may suppose that $\edge(D_1^*)$ lies to the right of or in line with $\edge(D_2^*)$, that $D_1^*$ has source $\0$ and width $w_1$, and that $D_2^*$ has source $(a,b)$ and width $w_2$. (Assuming $\0$ for the source of $D_1^*$ is permissible because we shall not assume anything about the location of lattice points inside the Duarte regions.) Define the new Duarte region $D^*$ as follows. Let $D^*$ have width $w$, where
\begin{equation}\label{eq:defw}
f(w) = f(w_1) + f(w_2) + 1,
\end{equation}
and source $(w_1-w,h_0)$, where $h_0 := (h_t - h_b)/2$, and
\[
h_t := \max\big\{ b + f(w_2) - f(w_2 + a), \, 0 \big\}
\]
and
\[
h_b := \max\big\{ -b + f(w_2) - f(w_2 + a), \, 0 \big\}.
\]
(By convention, we set $f(x)=0$ if $x<0$.) Thus, $h_t$ is the distance between the top of $\edge(D_2^*)$ and the top-most point of $D_1^*$, provided the former point lies above the latter point, and similarly for $h_b$ with `top' replaced by `bottom'. Moreover, $\edge(D_1^*)$ and $\edge(D^*)$ lie on the same vertical line in $\R^2$. An example is shown in Figure~\ref{fig:subadd}.

Since the height condition $h(D^*) \leq h(D_1^*) + h(D_2^*) + 1$ follows immediately from~\eqref{eq:defw} (in fact, with equality), to prove the lemma it is enough to show that $D_1^* \cup D_2^* \subset D^*$, and therefore by Observation~\ref{obs:edges}, it suffices to show that $\edge(D_1^*) \cup \edge(D_2^*) \subset D^*$. We may assume that $\max\{h_t,h_b\} > 0$, since otherwise $\edge(D_2^*) \subset D_1^*$, which implies $D_2^* \subset D_1^*$ by Observation~\ref{obs:edges}, and in this case the lemma is a triviality.

Beginning with $D_1^*$, we shall show that in fact $\edge(D_1^*) \subset \edge(D^*)$. Without loss of generality let $h_0 \geq 0$, and observe that the vertical coordinates of the bottom-most points of $\edge(D_1^*)$ and $\edge(D^*)$ are $-f(w_1)$ and $h_0 - f(w)$ respectively. Since the source of $D^*$ is defined so that $\edge(D_1^*)$ and $\edge(D^*)$ lie in the same vertical line, it is enough to show that $h_0 - f(w) \leq -f(w_1)$. By~\eqref{eq:defw}, this is equivalent to $h_0 \leq f(w_2) + 1$. Now, since $D_1 \subset D_1^*$ is strongly connected to $D_2 \subset D_2^*$, we have
\[
b - \lfloor f(w_2) \rfloor - \lfloor f(w_2 + a) \rfloor \leq 2,
\]
by comparing the bottom-most point of $\edge(D_2^*)$ with the boundary of $D_1^*$. (Note that if $D_2^*$ lies entirely to the left of $D_1^*$ then we actually have the stronger inequality $b - \lfloor f(w_2) \rfloor \leq 1$.) Therefore,
\[
h_t = b + f(w_2) - f(w_2 + a) \leq 2 f(w_2) + 2.
\]
Thus if $h_b=0$ then $h_0 = h_t/2 \leq f(w_2) +1$ as required. If $h_b>0$ then
\[
h_0 = b < f(w_2) - f(w_2 + a) < f(w_2) + 1,
\]
so we are again done.

Now we move on to $D_2^*$. Once again, by Observation~\ref{obs:edges} it is enough to prove that $\edge(D_2^*) \subset D^*$, and so by symmetry (we are no longer assuming $h_0 \geq 0$) we only have to show that
\[
b + f(w_2) \leq h_0 + f(w_2 + a - w_1 + w);
\]
that is, we have to show that the vertical coordinate of the top-most point of $\edge(D_2^*)$ is at most that of the upper boundary point of $D^*$ in the same vertical line. If $h_t>0$ and $h_b>0$ then $h_0=b$ and we are done by the monotonicity of $f$. Here we are using the fact that $w + a \geq w_1$, which is obtained by observing that $a \geq -w_2 - 1$, since $D_1$ and $D_2$ are strongly connected, and then by observing that $w \geq w_1 + w_2 + 1$ follows from $f(w) = f(w_1) + f(w_2) + 1$ by the convexity of $f$. If $h_b=0$ then $h_0 = b + f(w_2) - f(w_2 + a)$, so are we again easily done by the monotonicity of $f$. Finally, if $h_t=0$ then $b\leq 0$ and we are done as before.
\end{proof}

We can now deduce the following extremal result for internally spanned droplets.

\begin{lemma}\label{lem:extremal}
Let $K \subset \Z^2$ be a finite set such that $[K]$ is strongly connected. Then
\[
h\big( D([K]) \big) \leq 2|K| - 1.
\]
In particular, if $D$ is an internally spanned droplet, then
\[
|D\cap A| \geq \frac{h(D) + 1}{2}.
\]
\end{lemma}

\begin{proof}
The first assertion follows by induction on $|K|$ from Lemma~\ref{lem:subadd} and the spanning algorithm. Indeed, if $|K| = 1$ then $h\big( D([K]) \big) = 1$, as required, so assume that $|K| \ge 2$ and assume that the result holds for all proper subsets of $K$.

Run the spanning algorithm with initial set $K$, and observe that, since $[K]$ is strongly connected, we have $\< K \> = \{ D([K] \}$. Let $\K^{T-1} = \{K_1, K_2\}$, and observe that $[K_1]$, $[K_2]$ and $[K_1 \cup K_2]$ are strongly connected and $|K_1| + |K_2| = |K|$. Therefore, by the induction hypothesis and Lemma~\ref{lem:subadd}, we have
\begin{align*}
h\big( D([K]) \big) & = h\big( D([K_1\cup K_2]) \big) \le h\big( D([K_1]) \big) + h\big( D([K_2]) \big) + 1\\
& \le \big( 2|K_1| - 1 \big) + \big( 2|K_1| - 1 \big) + 1 = 2|K| - 1,
\end{align*}
as required.

The second assertion of the lemma follows from the first after noting that if $D$ is internally spanned then there exists a set $K \subset D\cap A$ such that $[K]$ is strongly connected and $D\big([K]\big) = D$, by Lemma~\ref{lem:span}.
\end{proof}

We will use Lemma~\ref{lem:extremal} in Section~\ref{sec:small} to deduce a non-trivial bound on the probability that a very small droplet is internally spanned, and hence prove the base case in our application of the method of iterated hierarchies. 

Our next lemma implies that critical droplets exist, and is based on a fundamental observation of Aizenman and Lebowitz~\cite{AL}, which has become a standard tool in the study of bootstrap percolation. In order to obtain a sharp threshold for the Duarte model, we will need the following, slightly stronger variant of their result.

\begin{lemma}\label{lem:critical:droplet}
If $[A] = \Z_n^2$, then there exists a pair $(D_1,D_2)$ of disjointly internally spanned droplets such that
\[
\max\big\{ h(D_1), h(D_2) \big\} \le \frac{1-\eps}{p} \log \frac{1}{p} \quad \text{and} \quad h(D_1) + h(D_2) \ge \frac{1-\eps}{p} \log \frac{1}{p} - 1,
\]
and $d(D_1,D_2) \le 2$. 
\end{lemma}

\begin{proof}
Run the spanning algorithm, starting with $S = A$, until the first time $t$ at which there exists a set $K \in \K^t$ that spans a droplet $D(K)$ of height larger than $(1-\eps)p^{-1} \log 1/p$. Since $K$ was created in step $t$, it follows that $K = K_1 \cup K_2$, where $K_1, K_2 \in \K^{t-1}$ are disjoint subsets of $A$ such that $[K_1]$ and $[K_2]$ are both strongly connected. Setting $D_1 = D\big( [K_1] \big)$ and $D_2 = D\big( [K_2] \big)$, we have 
\[
\max\big\{ h(D_1), h(D_2) \big\} \le \frac{1-\eps}{p} \log \frac{1}{p} \quad \text{and} \quad h\big(D(K) \big) > \frac{1-\eps}{p} \log \frac{1}{p}
\]
by our choice of $t$, and $D_1$ and $D_2$ are disjointly internally spanned by $K_1$ and $K_2$, respectively. By Lemma~\ref{lem:subadd}, it follows that
\[
h(D_1) + h(D_2) \ge \frac{1-\eps}{p} \log \frac{1}{p} - 1,
\]
as required.
\end{proof}

We will also need the following variant of Lemma~\ref{lem:critical:droplet}, which is closer to the original lemma of Aizenman and Lebowitz. Since the proof is so similar to that of Lemma~\ref{lem:critical:droplet}, it is omitted.

\begin{lemma}\label{lem:AL}
Let $D$ be an internally spanned droplet. Then for any $1\leq k\leq h(D)$, there exists an internally spanned droplet $D'\subset D$ such that $k\leq h(D') \leq 2k$. \qedhere
\end{lemma}

\subsection{Hierarchies}\label{sec:hierarchies}

In this section we will recall the definition and some basic properties of \emph{hierarchies}, which were introduced in~\cite{Hol} and subsequently used and developed by many authors, for example in \cite{BBDM,BBM3d,BDMS,DE,DH,GHM}. We will be quite brief, and refer the reader to~\cite{BDMS} for more details.

\begin{definition}\label{def:hierarchy}
Let $D$ be a droplet. A \emph{hierarchy} $\hier$ for $D$ is an ordered pair $\hier=(G_\hier,D_\hier)$, where $G_\hier$ is a directed rooted tree such that all of its edges are directed away from the root $\root$, and $D_\hier \colon V(G_\hier) \to \mathcal{P}(\Z^2)$ is a function that assigns
%\footnote{Note that a droplet can be specified by its height and left endpoint.}
% -- this is not true: height is an interger (width and left endpoint would work, but why go to the trouble?)
to each vertex of $G_\hier$ a droplet, such that the following conditions are satisfied:
\begin{enumerate}
\item the root vertex corresponds to $D$, so $D_\hier(\root) = D$;
\item each vertex has out-degree at most 2;
\item if $v \in N_{G_\hier}^\to(u)$ then $D_\hier(v) \subset D_\hier(u)$;
\item if $N_{G_\hier}^\to(u) = \{v,w\}$ then $D_\hier(u) \in \< D_\hier(v)\cup D_\hier(w) \>$.
\end{enumerate}
\end{definition}
Condition (4) is equivalent to the statement that $D_\hier(v) \cup D_\hier(w)$ is strongly connected and that $D_\hier(u)$ is the smallest droplet containing their union. We usually abbreviate $D_\hier(u)$ to $D_u$. 

%The next definition controls the absolute and relative sizes of the droplets corresponding to vertices of $G_\hier$, which in turn allows us to control the number of hierarchies. 

\begin{definition}\label{def:hier2}
Let $t > 0$. A hierarchy $\hier$ for a droplet $D$ is \emph{$t$-good} if it satisfies the following conditions for each $u \in V(G_\hier)$:
\begin{enumerate}
\setcounter{enumi}{4}
\item $u$ is a leaf if and only if $t \le h(D_u) \le 2t$;
\item if $N_{G_\hier}^\to(u) = \{v\}$ and $|N_{G_\hier}^\to(v)| = 1$ then
\[
t \leq h(D_u) - h(D_v) \leq 2t;
\]
\item if $N_{G_\hier}^\to(u)=\{v\}$ and $|N_{G_\hier}^\to(v)| \ne 1$ then $h(D_u) - h(D_v) \leq 2t$;
\item if $N_{G_\hier}^\to(u)=\{v,w\}$ then $h(D_u) - h(D_v) \geq t$.
\end{enumerate}
\end{definition}

%Next we relate the abstract family of good hierarchies defined above to the initial set $A$ of infected sites and to the $\U$-bootstrap process. %Note that $\Delta(D,D')$ is an increasing event.
%If $s = t$ (as it will in our applications), then we will simply say that $\hier$ is \emph{$t$-good}. 

The final two conditions, which we define next, ensure that a good hierarchy for an internally spanned droplet $D$ accurately represents the growth of the initial sites $D \cap A$. Given nested droplets $D \subset D'$, we define
\[
%begin{equation}\label{eq:defDelta}
\Delta(D,D') := \big\{ D' \in \< D \cup (D' \cap A) \> \big\}. 
\]
%end{equation}

\begin{definition}
A hierarchy $\hier$ for $D$ is \emph{satisfied} by $A$ if the following events all occur \emph{disjointly}:
\begin{enumerate}
\setcounter{enumi}{8}
\item if $v$ is a leaf then $D_v$ is internally spanned by $A$;
\item if $N_{G_\hier}^\to(u) = \{v\}$ then $\Delta(D_v,D_u)$ occurs.
\end{enumerate}
\end{definition}

Let us also make an easy observation about the event $\Delta(D,D')$, which will be useful in the next section.
\begin{obs}\label{obs:Delta}
Let $D \subset D_1 \subset D'$ be droplets. Then $\Delta(D,D')$ implies $\Delta(D_1,D')$.
\end{obs}

Next we recall some standard properties of hierarchies. Our first lemma motivates the definitions above by showing that every internally spanned droplet has at least one good and satisfied hierarchy. The proof is almost identical to Lemma~8.8 of~\cite{BDMS} (see also Propositions~31 and~33 of~\cite{Hol}), and so we omit it.

\begin{lemma}\label{lem:goodsat}
Let $t > 0$, and let $D$ be a droplet with $h(D) \ge t$ that is internally spanned by $A$. Then there exists a $t$-good and satisfied hierarchy for $D$. \qed
\end{lemma}

The next lemma allows us to bound $\P_p\big( \ispan(D) \big)$ in terms of the good and satisfied hierarchies of $D$. Let us write $\hier_D(t)$ for the set of all $t$-good hierarchies for $D$, and $L(\hier)$ for the set of leaves of $G_\hier$. We write $\prod_{u \to v}$ for the product over all pairs $\{u,v\} \subset V(G_\hier)$ such that $N_{G_\hier}^\to(u) = \{v\}$.

\begin{lemma}\label{lem:boundoverH}
Let $t > 0$, and let $D$ be a droplet. Then
%\begin{equation}\label{eq:boundoverH}
\[
\P_p\big(\ispan(D)\big) \leq \sum_{\hier \in \hier_D(t)} \bigg( \prod_{u \in L(\hier)} \P_p\big(\ispan(D_u)\big) \bigg)\bigg( \prod_{u \to v} \P_p\big(\Delta(D_v,D_u)\big) \bigg).
\]
%\end{equation}
\end{lemma}

\begin{proof}[Proof of Lemma~\ref{lem:boundoverH}]
Since the events $\ispan(D_u)$ for $u\in L(\hier)$ and $\Delta(D_v,D_u)$ for $u\to v$ are increasing and occur disjointly, this is an immediate consequence of Lemma~\ref{lem:goodsat} and the van den Berg--Kesten inequality.
\end{proof}

The following is little more than an observation, but we record it here for completeness.

\begin{lemma}\label{lem:sumofheights}
Let $\hier \in \hier_D(t)$. Then
\begin{equation}\label{eq:sumofheights}
\sum_{u \in L(\hier)} h(D_u) + \sum_{u\to v} \big( h(D_u)-h(D_v) \big) \geq h(D) - v(\hier).
\end{equation}
\end{lemma}

\begin{proof}
%The proof is by induction on $v(\hier)$.
Each vertex of out-degree $2$ in $G_\hier$ contributes an additive `error' of $1$ to the difference between $h(D)$ and the left-hand side of~\eqref{eq:sumofheights}, because of the application of Lemma~\ref{lem:extremal}. Vertices of out-degree $1$ in $G_\hier$ do not contribute any error. Thus~\eqref{eq:sumofheights} holds (and one could in fact replace $v(\hier)$ on the right-hand side of~\eqref{eq:sumofheights} with the number of vertices in $G_\hier$ of out-degree $2$).
\end{proof}

If $\hier \in \hier_D(t)$ is a hierarchy and $v \in L(\hier)$, then we say that $D_v$ is a \emph{seed} of $\hier$. We finish the section with the following easy lemma, cf.~\cite[Lemma~8.11]{BDMS}. 

\begin{lemma}\label{lem:numberofH}
Let $D$ be a droplet with $h(D) \leq p^{-1} \log 1/p$. Then there are at most
\begin{equation}\label{eq:numberofhierarchies}
\exp\Bigg( O\bigg( \frac{\ell \cdot h(D)}{t} \log \frac{1}{p} \bigg) \Bigg)
\end{equation}
$t$-good hierarchies for $D$ that have exactly $\ell$ seeds. Moreover, if $\hier$ is a $t$-good hierarchy for $D$ with $\ell$ seeds, then
\begin{equation}\label{eq:numberofvertices}
\big|V(\hier)\big| = O\bigg( \frac{\ell \cdot h(D)}{t} \bigg).
\end{equation}
\end{lemma}

\begin{proof}
The height of a $t$-good hierarchy for $D$ is at most $2 h(D) / t$, so the bound~\eqref{eq:numberofvertices} is straightforward. To count the number of choices of the droplet $D_u$ associated with the vertex $u$, we use Lemma~\ref{lem:numdroplets}. Thus, given integers $a$ and $b$ such that the source of $D_u$ lies in the square $(a,a+1] \times (b,b+1]$, and given $\lfloor w(D_u) \rfloor = w$, we have at most $w^{O(1)}$ choices for $D_u$, by Lemma~\ref{lem:numdroplets}. Summing over $a$, $b$ and $w$ gives at most $p^{-O(1)}$ choices in total for $D_u$, since there are at most $p^{-O(1)}$ choices for each of $a$, $b$ and $w$ by the condition on $h(D)$. The bound~\eqref{eq:numberofhierarchies} now follows.
\end{proof}

%We denote by $\ell(\hier)$ the number of seeds in a hierarchy $\hier$, and by $\hier_R^\ell(t)$ the set of all $t$-good hierarchies for $D$ that have exactly $\ell$ seeds. 

\section{Crossings}\label{sec:crossings}

Our aim in this section is to derive bounds on the probabilities of \emph{crossing events}, a phrase that we use informally to mean events of the form $\Delta(D,D')$, for droplets $D\subset D'$. The bounds we obtain will be used both to prove the inductive step $\ih(k)\Rightarrow\ih(k-1)$, for each $k\geq 1$, in Section~\ref{sec:small}, and the deduction of Proposition~\ref{prop:lower} from $\ih(0)$, in Section~\ref{sec:large}. The culmination of this section is the following lemma. Recall that $\eps_k = \eps^2 \cdot (3/4)^k$, where $\eps > 0$ is sufficiently small.

\begin{lemma}\label{lem:delta}
Let $k\geq 0$ and let $D\subset D'$ be droplets such that $h(D) \geq \eps_k^{-5}$,
\[
\eps_k^{-6} \leq h(D')-h(D) \leq \frac{p^{-(2/3)^k}}{2\log 1/p},
\]
and
\[
h(D') \leq \begin{cases} p^{-(2/3)^{(k-1)}}(\log 1/p)^{-1} & \text{if } k\geq 1, \\ (1-\eps)p^{-1}\log 1/p & \text{if } k=0. \end{cases}
\]
Suppose also that $\ih(k)$ holds. Then
\begin{equation}\label{eq:delta}
\P_p\big(\Delta(D,D')\big) \leq \exp \Bigg( -\bigg(\frac{1-1.1\eps_k}{2}\bigg)\left(\log\frac{1}{p} - p h(D')\right)\Big(h(D')-h(D)\Big) \Bigg).
\end{equation}
\end{lemma}

Observe that, while $k\geq 1$ and $h(D')\ll p^{-1}\log 1/p$, which will be the case throughout Section~\ref{sec:small}, the bound~\eqref{eq:delta} says (roughly) that
\[
\P_p\big(\Delta(D,D')\big) \apprle p^{(1-1.1\eps_k)(h(D')-h(D))/2}.
\]
The contribution from $-ph(D')$ in the exponent in~\eqref{eq:delta} only starts to matter when $k=0$ and the droplet approaches the critical size. However, it then plays a very important role: it is the reason why the exponents in~\eqref{eq:lower} and~\eqref{eq:ihbound} differ by a factor of $2$ (see the discussion after Definition~\ref{def:ih}). 

% is the source of one of the three factors of $1/2$ in the critical probability, Theorem~\ref{thm:Duarte}.

Lemma~\ref{lem:delta} is a relatively straightforward consequence of the following lemma about `vertical crossings'. Recall that
\[
f(x) := \frac{1}{2p}\log\left(1+\frac{\eps^3 px}{\log 1/p}\right)
\]
and that $\edge(D)$ denotes the right-hand side of a droplet. We will write $D_1 \sqsubset D_2$ to denote that $\edge\big(D^*(D_1)\big) \subset \edge\big(D^*(D_2)\big)$ holds.\footnote{Note that $\edge(D_1) \subset \edge(D_2)$ does not imply $\edge\big(D^*(D_1)\big) \subset \edge\big(D^*(D_2)\big)$, but that $\edge(D_1) \subset \edge(D_2)$ and $D_1 \subset D_2$ does.}

\begin{lemma}\label{lem:crossings}
Let $k\geq 0$ and let $D\sqsubset D'$ be droplets such that $h(D) \geq \eps_k^{-5}$ and
\[
\eps_k^{-5} \leq y := h(D') - h(D) \leq \frac{p^{-(2/3)^k}}{2\log 1/p}.
\]
Suppose also that $\ih(k)$ holds. Then
\begin{equation}\label{eq:crossings}
\P_p\big( \Delta(D,D') \big) \leq w(D')^{O(\eps_k^3 y)} \cdot \bigg(\frac{p}{f'\big(w(D')\big)}\bigg)^{\left(1-1.01\eps_k\right)y/2}.
\end{equation}
\end{lemma}

We reiterate at this point that the constant implied by the $O(\cdot)$ notation in the statement of the lemma is absolute: that is, it does not depend on any other parameter (in particular, it does not depend on $\eps$ or $k$). (In fact, one could take the constant to be 10, but we choose not to keep track of this.)

In order to prove Lemma~\ref{lem:crossings} we shall examine how growth from $D$ to $D'$ could occur. To do this, we shall show inductively that there exists a sequence of nested droplets $D=D_0\sqsubset\dots\sqsubset D_m=D'$ such that, for each $1\leq i\leq m-1$, either $(D_i\setminus D_{i-1})\cap A$ contains a large constant number of relatively `densely spaced' sites (an event which we think of, informally, as corresponding to the droplet growing row-by-row), or it spans a `saver' droplet of at least a large constant size. These alternatives are defined precisely in Definition~\ref{def:partition}.

In order to state that definition, we will need a weaker notion of connectivity than the strong connectivity used in conjunction with spanning, which will enable us to say what we meant by `relatively densely spaced' in the previous paragraph. Very roughly speaking, we say that a small set of sites is `weakly connected and $D$-rooted', for some droplet $D$, if the sites (might) help $D$ to grow vertically `faster than it should'.
% The parameter $w$ is a measure of how sparcely spaced the sites can be.

Henceforth in this section let us fix $k \geq 0$ and let $p>0$ (and hence $f'(0)$) be sufficiently small. %Recall that $\eps_k = \eps \cdot (2/3)^k$.

\begin{definition}\phantomsection\label{def:weak}
\begin{itemize}
\item[$(a)$] Define a relation $\prec$ on $\Z^2$, called the \emph{weak relation}, as follows. Given sites $x=(a_1,b_1)$ and $y=(a_2,b_2)$, we say that $x \prec y$ if
\[
a_2 - a_1 \geq -\eps_k^{-6} \qquad \text{and} \qquad |b_2 - b_1| \leq 2.
\]
%or, $(b_1-b_2)(b_2-b_0)>0$ and $-2\eps_k^{-6} \leq a_2-a_1 \leq 0$.\smallskip
\item[$(b)$] We say that a finite set $Y\subset\Z^2$ is \emph{weakly connected} if the graph on $Y$ with edge set $\big\{xy \in Y^{(2)} : x \prec y \text{ or } y \prec x \big\}$ is connected. \smallskip
\item[$(c)$] Now let $D$ be a droplet, with width $w$ and source $\big(a_0, b_0\big)$, and let $Z_D:=\big\{(a,b)\in\Z^2\setminus D : a\leq a_0 + w \big\}$. A weakly connected set $Y \subset Z_D$ is \emph{$D$-rooted} if for every $y\in Y$ there exist $y_1,\dots,y_j\in Y$ (for some $j\geq 0$) and $x\in D$ such that
\[
x \prec y_1 \prec y_2 \prec \ldots \prec y_j \prec y.
\]
%and $x_0$ is strongly connected to $D$.
The site $x$ is called a \emph{root} for $y$ with respect to $D$.
\end{itemize}
\end{definition}

The following lemma elucidates the key property of the definition above. The somewhat verbose statement (in terms of the numbers $h_1$ and $h_2$) is needed because in the applications we do not want the final bound in~\eqref{eq:gamma} to depend on $|Y|$, which may be much larger than $h_1 + h_2$.

\begin{lemma}\label{lem:atleastgamma}
Let $h_1,h_2 \ge 0$ and let $p > 0$ be sufficiently small. Now let $D$ be a droplet with width $w$ and source $(a_0,b_0)$, let $Y\subset\Z^2\setminus D$ be a finite set, and partition $Y$ into $Y^{(1)}\cup Y^{(2)}$, where $Y^{(1)} := \big\{ (a,b)\in Y : b\geq b_0 \big\}$ and $Y^{(2)} := Y\setminus Y^{(1)}$. Suppose that for each $y\in Y$ there exists a weakly connected and $D$-rooted set $Y'\subset Y$ containing $y$, such that $|Y' \cap Y^{(i)}| \leq h_i$ for $i=1,2$. Then
\begin{equation}\label{eq:gamma}
h\big( D(D\cup Y) \big) \leq h(D) + 2h_1 + 2h_2 + 2.
\end{equation}
\end{lemma}

\begin{proof}
Let us in fact set $\0$ to be the source of $D$. As in Lemma~\ref{lem:subadd}, this is permissible because we shall not need to assume that the lattice points inside $D$ have integer coordinates, neither shall we need to assume this about the elements of $Y$. Let $D^* := D^*(D)$, and let $D_1^*$ be the Duarte region with width $w_1$, where $f(w_1) = f(w) + h_1 + h_2 + 1$, and source $(w - w_1, h_0)$, where $h_0 := h_1 - h_2$. We claim that
\begin{equation}\label{eq:gammaclaim}
D^*\cup Y\subset D_1^*.
\end{equation}
Once we have this the lemma will follow, since
\[
h(D_1^*) = 2f(w_1) + 1 = 2f(w) + 2h_1 + 2h_2 + 3 = h(D^*) + 2h_1 + 2h_2 + 2.
\]

To show that $D^* \subset D_1^*$ it is enough to have $\edge(D^*) \subset D_1^*$, by Observation~\ref{obs:edges}. This containment would hold if $f(w_1) - f(w) \geq |h_0|$, since $\edge(D^*)$ is contained in the same vertical line in $\R^2$ as $\edge(D_1^*)$. But this inequality is immediate from the definitions of $w_1$ and $h_0$, so $D^* \subset D_1^*$ holds. The more substantive task is to show that $Y\subset D_1^*$, and for this the key observation is as follows.

\begin{claim}\label{clm:yinD1}
If $x=(a_1,b_1)\in D^*$ and $y=(a,b)\in\R^2\setminus D^*$ are such that
\[
a_1 - O(1) \leq a \leq w \qquad \text{and} \qquad -2h_2 \leq b - b_1 \leq 2h_1,
\]
then $y\in D_1^*$.
\end{claim}

\begin{figure}
  \centering
  \begin{tikzpicture}[>=latex,scale=1.3]
    % main (two) droplets
    \fill [gray!20!white] plot [domain=1:8] (\x,{ln(1+\x/10)}) -- plot [domain=8:1] (\x,2);
    \draw [domain=0:8] plot (\x,{ln(1+\x/10)}) node [right] {$D^*$};
    \draw [domain=0:8] plot (\x,{ln(10+\x/2)}) node [right] {$D_1^*$};
    \begin{scope}[every node/.style={fill,circle,inner sep=1pt}]
      \node [label=below:$x$] at (3,{ln(1.3)}) {};
      \node at (3,2) {};
      \node at (3,{ln(11.5)}) {};
      \node at (1,2) {};
    \end{scope}
    \draw [loosely dashed] (3,{ln(1.3)}) -- (3,{ln(11.5)});
    \draw [<->] (3.2,{ln(1.3)}) -- node [right] {$2h_1$} (3.2,2);
    \draw [<->] (3.2,2) -- node [right] {$\geq 1$} (3.2,{ln(11.5)});
%    \draw [loosely dashed] (3,2) -- (1,2);
    \draw [<->] (3,1.5) -- node [below] {$O(1)$} (1,1.5);
    \draw [densely dashed] (1,{ln(1.1)}) -- (1,2) -- (8,2);
%    \draw [domain=-1:1,variable=\t] plot (8,{1.3*ln(9)*\t});
%    \node at (6.5,0) {$D=D_0$};
%    \pgfmathsetmacro{\x}{exp(1/1.3)-1};
%    \draw (\x,-1) -- (\x,-2.4) node [below] {$D'=D_4$};   
  \end{tikzpicture}
  \caption{Claim~\ref{clm:yinD1} asserts that the shaded region is contained in $D_1^*$. The essence of the proof is that the vertical distance between the boundaries of $D^*$ and $D_1^*$ is always at least $2h_1 + 1$, and $p$ (and hence $f'(0)$) can be taken sufficiently small to beat the $O(1)$ distance the region extends to the left of $x$.}
  \label{fig:yinD1}
\end{figure}
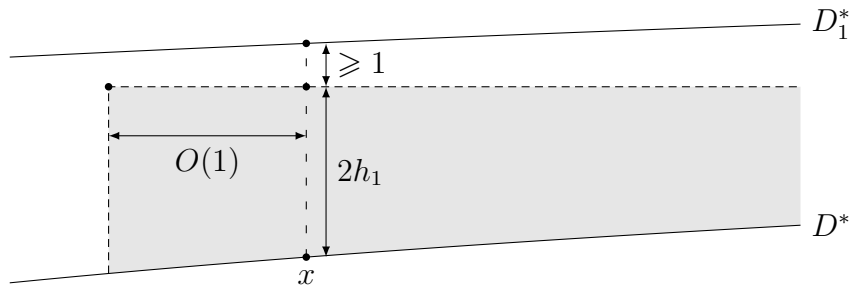

\begin{clmproof}{clm:yinD1}
This follows essentially from the convexity of $f$ and the fact that $p$ (and hence $f'(0)$) is sufficiently small. The key is that the top of $D_1^*$ always passes at least distance $2h_1 + 1$ above $x$.

To spell out the details, first let us assume by symmetry that $b \geq 0$, and observe that for each $t \in [0,w]$ we have
\[
h_0 + f(t - w + w_1) - f(t) \geq h_0 + f(w_1) - f(w) = 2h_1 + 1,
\]
where we used the convexity of $f$ for the inequality. But the left-hand side is the difference between the vertical coordinates of the top-most points in $D_1^*$ and $D^*$, intersected with the column with horizontal coordinate $t$. Thus we are done if $a=a_1$. If $a>a_1$ then we are also done, since $f$ is increasing. Finally, if $a_1 - O(1) \leq a < a_1$ then we are again done, this time since $p$ is sufficiently small and hence $f'(t)$ is sufficiently small for all $t>0$.
\end{clmproof}

To complete the proof of the lemma, recall that we wish to show $Y \subset D_1^*$. Let $y=(a,b)\in Y$ and without loss of generality let us assume $y\in Y^{(1)}$. We know by the condition of the lemma that there exists a weakly connected and $D$-rooted set $Y' \subset Y$ containing $y$, such that $|Y' \cap Y^{(i)}| \leq h_i$ for $i=1,2$. Now take a path of sites
\[
x \prec y_1 \prec \ldots \prec y_j \prec y,
\]
with $j\geq 0$, such that $\{y_1,\dots,y_j,y\} \subset Y' \cap Y^{(1)}$, and such that either $x\in D$ or $b_1 \in [-2,0)$, where $x=(a_1,b_1)$. To construct such a path, first allow the $y_i$ to belong to $Y'$, then, starting at $y$, truncate the path if necessary at the first element having negative vertical coordinate. It follows that $j + 1 \leq |Y' \cap Y^{(1)}| \leq h_1$.

If $x\in D$ then, by the definition of $\prec$, we have $b \leq b_1 + 2(j + 1) \leq b_1 + 2h_1$ and $a \geq a_1 - O(j) = a_1 - O(1)$. Hence $x$ and $y$ satisfy the conditions of Claim~\ref{clm:yinD1}. On the other hand, if $b_1 \in [-2,0)$ then $b \leq 2(j+1) \leq 2h_1$ and $a \geq - O(1)$. Hence in this case $\0$ and $y$ satisfy the conditions of the claim. In either case it follows that $y \in D_1^*$, and the proof is complete.
\end{proof}

We are now ready to make the key definition of the section, that of a \emph{satisfied partition} of a pair of droplets $D\sqsubset D'$. Let us fix $\gamma := \lfloor \eps_k^{-3} / 2 \rfloor$. 

\begin{definition}\label{def:partition}
Let $D\sqsubset D'$ be droplets. A \emph{satisfied partition $\PP$ of $(D,D')$} is a sequence of droplets $\PP=(D_i)_{i=0}^m$, for some $m\geq 1$, such that
\[
D=D_0 \sqsubset D_1 \sqsubset \cdots \sqsubset D_m=D',
\]
$h(D_m)-h(D_{m-1}) \leq 5\gamma$, and for each $1\leq i\leq m-1$, we have $h(D_i)-h(D_{i-1})\geq 2\gamma$ and (at least) one of the following events occurs:
\begin{itemize}
\item[(1)] $h(D_i)-h(D_{i-1}) \leq 2\gamma + 2$ and $\big(D_i\setminus D_{i-1}\big)\cap A$ contains a weakly connected $D_{i-1}$-rooted set of size at least $\gamma$.\smallskip
\item[(2)] There exists a droplet $S_i$ spanned\footnote{Recall that $S_i$ is spanned by a set $K$ if there exists $K' \subset K$ such that $S_i \in \<K'\>$. Note that here it need not necessarily be the case that $S_i \subset D_i\setminus D_{i-1}$.} by $(D_i\setminus D_{i-1})\cap A$, with
\begin{equation}\label{eq:satisfied}
w(S_i) \geq \eps_k^{-6}-1 \qquad \text{and} \qquad h(S_i) \geq h(D_i)-h(D_{i-1}) - \eps_k^{-3},
\end{equation}
and such that either $h(S_i) \geq \eps_k^{-5}$ or the rightmost $\eps_k^{-6}-1$ columns of $S_i$ all contain an element of $(D_i\setminus D_{i-1})\cap A$.
%\footnote{Note that the definition of the width of a droplet is such that a droplet having width $w$ may intersect $w+1$ columns of $\Z^2$.}
(We call $S_i$ a \emph{saver droplet}.)
\end{itemize}
\end{definition}

\begin{figure}
  \centering
  \begin{tikzpicture}[>=latex,scale=1.3]
    % main (two) droplets
    \draw [domain=0:8] plot (\x,{1.3*ln(1+\x)});
    \draw [domain=0:8] plot (\x,{-1.3*ln(1+\x)});
    \draw [domain=-1:1,variable=\t] plot (8,{1.3*ln(9)*\t});
    \draw [domain=0:5] plot ({\x+3},{1.3*ln(1+\x)});
    \draw [domain=0:5] plot ({\x+3},{-1.3*ln(1+\x)});
    % intermediate droplets
    \draw [domain=0:5.7,densely dashed] plot ({\x+2.3},{1.3*ln(6.7/6)+1.3*ln(1+\x)});
    \draw [domain=0:5.7,densely dashed] plot ({\x+2.3},{1.3*ln(6.7/6)-1.3*ln(1+\x)});
    \draw [domain=0:6.6,densely dashed] plot ({\x+1.4},{1.3*ln(6.7*6.7/(6*7.6))+1.3*ln(1+\x)});
    \draw [domain=0:6.6,densely dashed] plot ({\x+1.4},{1.3*ln(6.7*6.7/(6*7.6))-1.3*ln(1+\x)});
    \draw [domain=0:7.2,densely dashed] plot ({\x+0.8},{1.3*ln(8.2/9)+1.3*ln(1+\x)});
    \draw [domain=0:7.2,densely dashed] plot ({\x+0.8},{1.3*ln(8.2/9)-1.3*ln(1+\x)});
    % savers
    \draw [domain=0:0.3] plot ({\x+3.05},{-1.1+1.3*ln(1+\x)});
    \draw [domain=0:0.3] plot ({\x+3.05},{-1.1-1.3*ln(1+\x)});
    \draw [domain=-1:1,variable=\t] plot (3.35,{-1.1+1.3*ln(1.3)*\t});
%    \draw [domain=0:0.15,very thick] plot ({\x+4},{2+1.3*ln(1+\x)});
%    \draw [domain=0:0.15,very thick] plot ({\x+4},{2-1.3*ln(1+\x)});
%    \draw [domain=-1:1,variable=\t,very thick] plot (4.15,{2+1.3*ln(1.15)*\t});
    % gamma sets
    \begin{scope}[every node/.style={cross out,draw,inner sep=0.05cm}]
    \foreach \x in {0,1,...,4}
      \node at ({6+0.2*\x},{1.9+0.1*\x}) {};
    \foreach \x in {0,1,...,4}
      \node at ({6.5+0.2*\x},{-2.4-0.06*\x}) {};
    \end{scope}
    \node at (6.5,0) {$D=D_0$};
%    \pgfmathsetmacro{\x}{exp(1.56/1.3)+1.3};
%    \draw (\x,-1.5) -- (\x,-2.4) node [below] {$D_1$};
%    \pgfmathsetmacro{\x}{exp(1.47/1.3)+0.4};
%    \draw (\x,-1.5) -- (\x,-2.4) node [below] {$D_2$};
%    \pgfmathsetmacro{\x}{exp(1.46/1.3)-0.2};
%    \draw (\x,-1.5) -- (\x,-2.4) node [below] {$D_3$};
    \pgfmathsetmacro{\x}{exp(1/1.3)-1};
    \draw (\x,-1) -- (\x,-2.4) node [below] {$D'=D_4$};   
  \end{tikzpicture}
  \caption{An example of a satisfied partition $\PP=(D_i)_{i=0}^4$ of $(D,D')$. The small droplet is a saver droplet and the clusters of five crosses are weakly connected sets, each $D_i$-rooted for some $i$. Thus, with $\gamma=5$, condition~(1) of Definition~\ref{def:partition} is satisfied when $i=1$ and $3$, and condition~(2) is satisfied when $i=2$.} % (although neither condition of the definition \emph{need} be satisfied when $i=4$).}
  \label{fig:partition}
\end{figure}
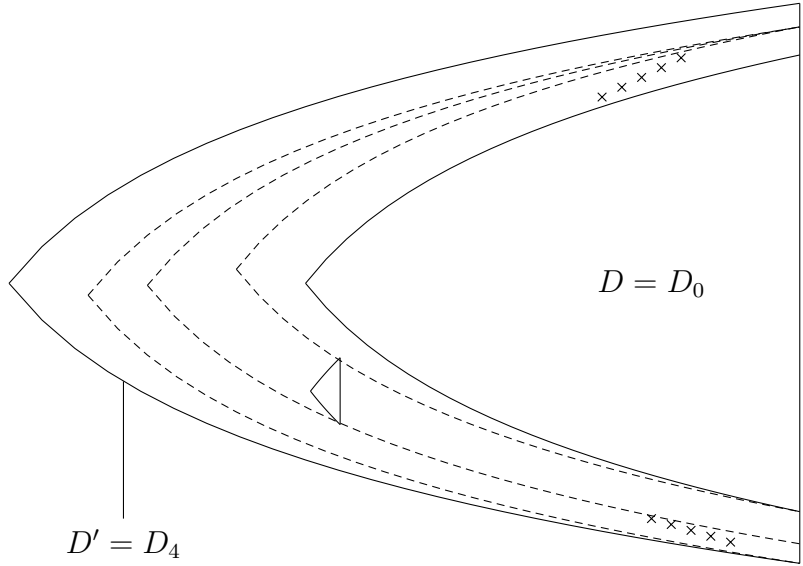

The next lemma, which states that the crossing event for droplets $D \sqsubset D'$ implies the existence of a satisfied partition for $(D,D')$, is the heart of the proof of Lemma~\ref{lem:delta}, and is the key deterministic tool in the proof of Theorem~\ref{thm:Duarte}.

\begin{lemma}\label{lem:partition}
Let $D \sqsubset D'$ be droplets with $h(D) \geq \eps_k^{-5}$. If $\Delta(D,D')$ holds then there exists a satisfied partition of $(D,D')$.
\end{lemma}

From here until the end of the proof of Lemma~\ref{lem:partition}, let us fix droplets $D\sqsubset D'$. %, and set $w:=w(D)$.
%Despite Definition~\ref{def:partition}~(1) mentioning $(D,w(D'))$-weakly connected sets, for convenience we shall only consider $(D,w)$-weakly connected sets: since $w(D') \geq w$ and $f'$ is decreasing, such sets are automatically $(D,w(D'))$-connected, so what we shall prove will actually be stronger.
Let $Y_1,\dots,Y_s$ be the collection of maximal weakly connected and $D$-rooted sets in $(D'\setminus D)\cap A$. (These sets are disjoint, since if $Y_i \cap Y_j \neq \emptyset$ then $Y_i \cup Y_j$ is weakly connected and $D$-rooted.\footnote{This is because the elements of a $D$-rooted set do not all have to have the same root.})
%\footnote{So for each $x_0\in(D'\setminus D)\cap A$ strongly connected to $D$, the maximal rooted $(D,w)$-weakly connected set containing $x_0$ is one of the $Y_i$.} 
Finally, let
\begin{equation}\label{eq:YZ}
Y:= Y_1\cup\dots\cup Y_s \qquad \text{and} \qquad Z:= [D\cup Y] \setminus D.
\end{equation}

% Before proving Lemma~\ref{lem:partition}, we need the following fact about weakly connected sets.

The first preliminary we need in the build-up to the proof of Lemma~\ref{lem:partition} is the following easy observation about elements of $Z$. %The observation holds by induction on $\big| [D\cup Y_i] \setminus D \big|$, for example.

\begin{obs}\label{obs:yi}
Let $z = (c,d) \in [D\cup Y_i] \setminus D$ for some $1\leq i\leq s$. Then one of the following holds, in each case with $a \leq c$ (and $a' \leq c$ in case $(b)$):
\begin{enumerate}
\item[$(a)$] there exists a site $y := (a,d) \in Y_i$;
\item[$(b)$] there exist sites $y := (a,d-1) \in Y_i$ and $y' := (a',d+1) \in Y_i$;
\item[$(c)$] $(c,d-1) \in D$ and there exists a site $y := (a,d+1) \in Y_i$;
\item[$(d)$] $(c,d+1) \in D$ and there exists a site $y := (a,d-1) \in Y_i$.
\end{enumerate}
\end{obs}

Next we need the following lemma, which says that we may obtain $D\cup Z$ from $D$ by taking the closures with each of the $Y_i$ independently. This will enable us to control the size of $Z$. The lemma also says that there is a good separation between $D\cup Z$ and the elements of $A$ outside of $D\cup Z$.

\begin{lemma}\label{lem:lessthangamma}
%Suppose $|Y_i|\leq\gamma-1$ for all $i \in [s]$. Then
We have
\begin{equation}\label{eq:partZ}
Z = \big( [D\cup Y_1] \cup \dots \cup [D\cup Y_s] \big) \setminus D.
\end{equation}
Moreover, if $x\in A\setminus (D\cup Z)$ then there does not exist $z\in D\cup Z$ such that $z \prec x$.
%\begin{equation}\label{eq:part1}
%a_2-a_1 \leq 2\eps_k^{-6} \qquad \text{and} \qquad |b_2-b_1| \leq 2,
%\end{equation}
%where $x=(a_1,b_1)$ and $z=(a_2,b_2)$.
\end{lemma}

\begin{proof}
To prove~\eqref{eq:partZ}, we shall show that no site $z_1 \in [D\cup Y_1] \setminus D$ is strongly connected to a site in $z_2 \in [D\cup Y_2] \setminus D$, unless $z_1$ and $z_2$ lie either side (vertically) of an element of $D$. This will establish the claim, since it would imply that the set $[D\cup Y_1] \cup [D\cup Y_2]$ is closed, and since the ordering of the $Y_i$ was arbitrary.

First, we make the following observation, which follows immediately from the definition of $\prec$:
\begin{enumerate}
\item[$(\ast)$] If $y_1 = (a_1,b_1) \in Y_1$ and $y_2 = (a_2,b_2) \in Y_2$, then, since neither $y_1 \prec y_2$ nor $y_2 \prec y_1$ holds, we must have $|b_1 - b_2| \geq 3$.
\end{enumerate}
Since $Y_1$ and $Y_2$ are each weakly connected, it follows (without loss of generality) that $\max\{ b: (a,b) \in Y_1 \} \le \min\{b : (a,b) \in Y_2 \} - 3$. Let $z_1 = (c_1,d_1)$ and $z_2 = (c_2,d_2)$, and suppose first that $d_1 \le \max\{ b: (a,b) \in Y_1 \}$. Then, since $z_1$ and $z_2$ are strongly connected, it follows that $d_2 < \min\{b : (a,b) \in Y_2 \}$. Now, by Observation~\ref{obs:yi}, it follows that $(c_2,d_2-1) \in D$ and $d_2 = \min\{b : (a,b) \in Y_2 \} - 1$. But since $z_1$ and $z_2$ are strongly connected, this implies that $c_1 = c_2$ and $d_1 = d_2 - 2$, and hence $z_1$ and $z_2$ lie either side (vertically) of an element of $D$, as claimed. The proof in the case $d_1 > \max\{ b: (a,b) \in Y_1 \}$ is identical. 

To see the second part of the lemma, let $x \in A \setminus (D \cup Z)$, and suppose that $z \prec x$ for some $z \in D\cup Z$. Observe that $z$ cannot be in $D$, because then $x$ would belong to one of the $Y_i$. So in fact we have $z\in Z$ and we may assume further that $x$ is not weakly connected to any element of $D$. By the first part of the lemma, we may also assume that $z \in [D \cup Y_1] \setminus D$. We shall show that there exists $y \in Y_1$ such that $y \prec x$, which would imply that $x$ belongs to $Y_1$, a contradiction.

Let $x = (c_0,d_0)$ and let $z = (c,d)$, and let $y$ (and possibly also $y'$) be the sites obtained from Observation~\ref{obs:yi} applied to $z$. If option~$(a)$ holds then we immediately have $y \prec x$. If option~$(b)$ holds then we take $y = (a,d+1)$ if $d_0 \geq d$, to obtain $y \prec x$, and we take $y' = (a,d-1)$ if $d_0 < d$, to obtain $y' \prec x$. Finally, if option~$(c)$ holds (say), then since $(c,d-1) \in D$ and $z \prec x$, we must have $d_0 \geq d-1$, and therefore we have $y \prec x$. (Here we have used the assumption that $x$ is not weakly connected to any element of $D$: if $z$ is near to the left-hand end of $D$, then there do exist sites in $\Z^2 \setminus D$ within horizontal distance $\eps_k^{-6}$ to the left of $z$, and having vertical coordinate 2 less than that of $z$. However, any such site is weakly connected to $D$.) This completes the proof of the second part of the lemma.
\end{proof}

We are now ready to prove Lemma~\ref{lem:partition}. The basic idea is as follows: if none of the sets $Y_i$ has size at least $\gamma$, then since (by~\eqref{eq:partZ}) we can obtain $D\cup Z$ from $D$ by taking the closure of $D$ with each of the $Y_i$ independently, we can control the size of each (strongly) connected component of $Z$. Since Definition~\ref{def:weak} ensures that there is a large region disjoint from $A$ around any maximal weakly connected component, the event $\Delta(D,D')$ allows us to deduce the existence of a saver droplet sufficiently large to penetrate through this region; see Claim~\ref{clm:saver} below.

\begin{proof}[Proof of Lemma~\ref{lem:partition}]
The proof is by induction on $\lfloor h \rfloor$, where $h := h(D') - h(D)$. When $h \leq 5\gamma$ there is nothing to prove: we may take $m=1$, $D_0=D$ and $D_1=D'$, so that $\PP=(D_0,D_1)$ trivially satisfies Definition~\ref{def:partition}. Thus we shall assume that $h > 5\gamma$ and that the result holds for smaller non-negative values of $\lfloor h \rfloor$.

Suppose first that $|Y_i| \geq \gamma$ for some $i$. In this case we will show that there exists a droplet $D \sqsubset D_1 \sqsubset D'$ with $1 \le h(D_1) - h(D) \le 2\gamma + 2$ and such that $\big(D_1\setminus D \big) \cap Y_i$ contains a weakly connected $D$-rooted set of size at least $\gamma$, as in Definition~\ref{def:partition}~$(1)$. In order to define $D_1$, we will first show that there exists a subset $W \subset Y_i$ with $|W| = \gamma$ that satisfies the conditions of Lemma~\ref{lem:atleastgamma}. Indeed, this follows by greedily adding points of $Y_i$ to $W$ one by one (starting from the empty set), maintaining the property that $W$ is $D$-rooted. (So a point $y \in Y_i$ may be added to $W$ if there exists $u \in D \cup W$ such that $u \prec y$.) It is easy to see that for each $u \in W$ there exists a set $W' \subset W$ with $u \in W'$ such that $W'$ is weakly connected and $D$-rooted (simply take the oriented path leading to $u$). Moreover, since $W' \subset W$ and $|W| = \gamma$, the conditions of Lemma~\ref{lem:atleastgamma} are satisfied for some $h_1,h_2 \ge 0$ with $h_1 + h_2 = \gamma$, and thus % Question: is 0 \in \N??
\[
h\big( D(D\cup W) \big) - h(D) \, \le \, 2\gamma + 2.
\]
If $h\big( D(D\cup W) \big) - h(D) \geq 2\gamma$, then set $D_1 = D(D\cup W)$; if not, then choose instead for $D_1$ any droplet such that $2\gamma \leq h(D_1) - h(D) \leq 2\gamma + 1$ and $D(D\cup W) \sqsubset D_1 \sqsubset D'$. In either case, the droplet $D_1$ satisfies the conditions of Definition~\ref{def:partition} with $i=1$ and $D_0 = D$. We may therefore apply induction to the pair $(D_1,D')$, noting that the event $\Delta(D_1,D')$ occurs by Observation~\ref{obs:Delta}, and, for the purpose of the induction on $\lfloor h \rfloor$, that we have ensured that $h(D') - h(D_1) \leq h(D') - h(D) - 1$.

Henceforth we shall assume that $|Y_i|\leq\gamma-1$ for each $1\leq i\leq s$. Our task is to show, using Lemma~\ref{lem:lessthangamma}, that there exists a saver droplet satisfying condition~(2) of Definition~\ref{def:partition}. In order to find the saver droplet, we begin by showing that either $[(D'\setminus D)\cap (A\setminus Y)]$ is strongly connected to $D \cup Z$, or we can take the whole of $D'$ to be the saver droplet.

\begin{claim}\label{clm:xz}
Either there exist sites $z \in D \cup Z$ and $x \in \big[ (D'\setminus D)\cap (A\setminus Y) \big]$ such that $z$ and $x$ are strongly connected, or we have
\begin{equation}\label{eq:supersaver}
D' \in \big\< (D'\setminus D)\cap (A\setminus Y) \big\>.
\end{equation}
\end{claim}

\begin{clmproof}{clm:xz}
Suppose~\eqref{eq:supersaver} does not hold. Firstly, note that
\begin{align}
\big[ D \cup (D'\cap A) \big] &= \Big[ D \cup Y \cup \big(D' \cap (A\setminus Y)\big) \Big] \notag \\
&= \Big[ \big( D \cup Z \big) \cup \big[ (D'\setminus D)\cap (A\setminus Y) \big] \Big], \label{eq:xz}
\end{align}
since $Y \subset D' \cap A$ and $D \cup Z = [D \cup Y]$. Secondly, the event $\Delta(D,D')$ implies that $[D \cup (D'\cap A)]$ contains a strongly connected set $L$ such that $D'=D(L)$. However, we cannot have $L \subset D \cup Z$, because if we apply Lemma~\ref{lem:atleastgamma} to the droplet $D$ and the set $Y$, with $h_1=h_2=\gamma$, then we obtain
\[
h\big(D(D\cup Z)\big) = h\big(D(D\cup Y)\big) \leq h(D) + 4\gamma + 2 < h(D'),
\]
where we have used the fact that $h(D') - h(D) > 5\gamma$. We also cannot have $L \subset \big[ (D'\setminus D)\cap (A\setminus Y) \big]$, because~\eqref{eq:supersaver} does not hold. Now, if the union of $D\cup Z$ and $\big[ (D'\setminus D)\cap (A\setminus Y) \big]$ is not closed, then we are done: this would immediately imply the existence of sites $x$ and $z$ as in the statement of the claim. If the union of the two sets is closed, then by~\eqref{eq:xz} we would have
\[
L \subset \big[ D \cup (D'\cap A) \big] = \big( D \cup Z \big) \cup \big[ (D'\setminus D)\cap (A\setminus Y) \big].
\]
Hence, since the strongly connected set $L$ is contained in neither $D\cup Z$ nor $\big[ (D'\setminus D)\cap (A\setminus Y) \big]$, it must intersect both, and therefore these sets must themselves be strongly connected, as required.
\end{clmproof}

We now have everything we need to find the saver droplet.

\begin{claim}\label{clm:saver}
There exists a droplet $S$ spanned by $(D'\setminus D) \cap A$ such that
\begin{equation}\label{eq:partD*}
w(S) \geq \eps_k^{-6} - 1 \qquad \text{and} \qquad h(S) \geq h\big(D(D\cup S)\big) - h(D) - \eps_k^{-3}.
\end{equation}
Moreover, either $h(S) \geq \eps_k^{-5}$, or the rightmost $\eps_k^{-6}-1$ columns of $S$ all contain an element of $(D'\setminus D) \cap A$.
\end{claim}

We will complete the proof of Lemma~\ref{lem:partition} after the proof of Claim~\ref{clm:saver}.

\begin{figure}
  \centering
  \begin{tikzpicture}[>=latex]
    \draw (0,0) -- (3,0) -- (3,0.5) -- (6,0.5) -- (6,1) -- (9,1) -- (9,1.5) -- (12,1.5); % -- (11,2) -- (12,2);
%    \draw (0,3) -- (2,3) -- (2,3.5) -- (5,3.5) -- (5,4) -- (8,4) -- (8,4.5) -- (11,4.5) -- (11,5) -- (12,5);
    \draw (0.5,0) rectangle (1,0.5) (1.5,1) rectangle (2,1.5) (3.5,1.5) rectangle (4,2) (8,2.5) rectangle (8.5,3);
    \draw [densely dashed] (1,0.5) -- (3,0.5) (2,1.5) -- (9,1.5) (4,2) -- (12,2) (9,2.5) (8.5,3) -- (12,3) (1.5,1) -- (1.5,0.5) (8,2.5) -- (8,2);
    \draw [densely dotted] (-1,0) -- (-1,1.5) -- (0,1.5) -- (0,2.5) -- (2,2.5) -- (2,3) -- (6.5,3) -- (6.5,4) -- (12,4);
    \draw [densely dashed] (8,3.5) rectangle (8.5,4);
    \draw (5.5,3.5) rectangle (6,4);
    \draw (4,3.5) -- (8.5,3.5) -- (8.5,5);
    \node at (10,0.75) {$D$};
    \node at (4.5,4) {$S$};
    \draw (5.75,3.75) -- (5.75,4.25) node [above] {$x'$};
    \draw (8.25,3.75) -- (8.25,4.25) node [above] {$x$};
    \draw (8.25,2.75) -- (8.75,2.75) node [right] {$z$};
  \end{tikzpicture}
  \caption{The setup in Claim~\ref{clm:saver}. The region below the solid line at the bottom of the figure is $D$; that above and to the left of the solid line at the top of the figure is $S$. Solid boxes are elements of $A$. The dashed lines bound the elements of the closure $Z = [D\cup Y]$. The dotted line bounds the set of sites weakly connected to $Y_i$. The sites $x$, $x'$ and $z$ are as in the claim. (Note that $x$ is not in $A$, so it is indicated by a dashed box. In this example we have $z\in A$, but that need not be the case; similarly, $x$ is shown as the bottom-right-hand element of $S$, which it need not be.)}
  \label{fig:claim}
\end{figure}
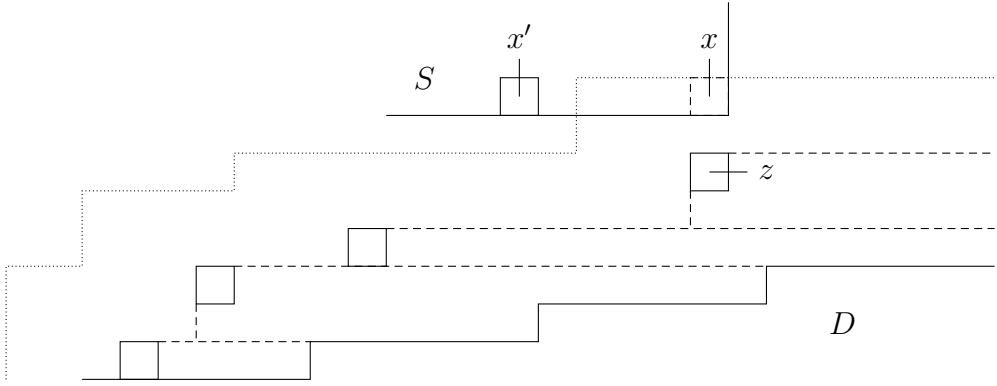

%\begin{proof}[Proof of Claim~\ref{clm:saver}]\let\qednow\qedsymbol\renewcommand{\qedsymbol}{}
\begin{clmproof}{clm:saver}
To begin, note that if~\eqref{eq:supersaver} holds then we may take $S=D'$, since then $h(S) = h(D') \geq \eps_k^{-5}$ by the assumption of Lemma~\ref{lem:partition}. So let us assume~\eqref{eq:supersaver} does not hold, and that therefore, by Claim~\ref{clm:xz}, there exist sites $z\in D\cup Z$ and $x\in [(D'\setminus D)\cap (A\setminus Y)]$ such that $z$ and $x$ are strongly connected. Without loss of generality, let us assume that in fact $z \in [D\cup Y_1]$, by~\eqref{eq:partZ}. Let $S \in \< (D'\setminus D)\cap (A\setminus Y) \>$ be the droplet spanned by the strongly connected component of $[(D'\setminus D)\cap (A\setminus Y)]$ containing $x$. We will show that $S$ is the desired droplet, i.e., that it has all of the claimed properties.

%Let $\D$ be the output of the spanning algorithm with input $(D'\setminus D)\cap (A\setminus Y)$, and let $S\in\D$ be the droplet spanned by the strongly connected component of $[(D'\setminus D)\cap (A\setminus Y)]$ containing $x$.

First we must show that the dimensions of $S$ satisfy the conditions of~\eqref{eq:partD*}. We begin with the height condition. If $z \in D$ then $S$ and $D$ are strongly connected, in which case
$$h\big( D( D \cup S ) \big) \, \le \, h(D) + h(S) +1,$$
by Lemma~\ref{lem:subadd}. So assume that $z \in [D\cup Y_1]\setminus D$, and let $D_z = D(C_z)$, where $C_z$ is the minimal column of (consecutive) sites containing $z$ and strongly connected to $D$. By the definition of the weak relation and the bound $|Y_1|\leq\gamma-1$, and since $p$ (and therefore $f'(0)$) is sufficiently small, it follows that $|C_z| \leq 2(\gamma - 1)$, and therefore $h\big( D(C_z) \big) \leq 2(\gamma - 1)$. Hence, by Lemma~\ref{lem:subadd} we have
\[
h\big( D(D\cup\{z\}) \big) \leq h(D) + h\big( D(C_z) \big) + 1 \leq h(D) + 2\gamma - 1.
\]
Now, since $z$ and $x$ are strongly connected, it follows again from Lemma~\ref{lem:subadd}, this time applied to droplets $D(D \cup \{z\})$ and $S$, that
\[ 
h\big( D(D\cup S) \big) \leq h\big( D(D\cup\{z\}) \big) + h(S) + 1 \leq h(D) + h(S) + 2\gamma.
\]
Since $2\gamma \le \eps_k^{-3}$, it follows that the height condition in~\eqref{eq:partD*} holds.

For the width condition in~\eqref{eq:partD*}, notice that since $x\in[S\cap A]$ (but $x\notin A$), at least one of the following must hold:
\begin{itemize}
\item $S\cap A$ has non-empty intersection with the row containing $x$;
\item $S\cap A$ has non-empty intersection with the row immediately above $x$ and the row immediately below $x$.
\end{itemize}
In either case, since $x$ and $z$ are strongly connected, there exists $x'\in S\cap A$ differing from $z$ in its vertical coordinate by at most $2$. Note moreover that we can choose $x'$ to be in the same strongly connected component of $[S\cap A]$ as $x$. Now since $x'\in A\setminus(D\cup Z)$, we cannot have $z\prec x'$, by Lemma~\ref{lem:lessthangamma}. Hence, writing $x'=(a_1,b_1)$ and $z=(a_3,b_3)$, it follows that $w(S) \ge a_3 - a_1 - 1 > \eps_k^{-6}$, which implies the claimed bound on $w(S)$.

Finally, we must show that the rightmost $\eps_k^{-6}-1$ columns of $S$ all contain an element of $(D'\setminus D)\cap A$. But this follows from the fact that $x$ and $x'$ lie in the same strongly connected component of $[S \cap A]$, using the bound $a_3-a_1 > \eps_k^{-6}$.
%\end{proof}
\end{clmproof}

We now finish the proof of Lemma~\ref{lem:partition}. Let $S$ be the (saver) droplet whose existence is guaranteed by Claim~\ref{clm:saver}. Set $D_1$ to be equal to $D(D\cup S)$, unless $h\big(D(D\cup S)\big) - h(D) < 2\gamma$, in which case instead set $D_1$ to be any droplet such that $D(D\cup S) \sqsubset D_1 \sqsubset D'$ and $2\gamma \leq h(D_1) - h(D) \leq 2\gamma + 1$ (cf.~the second paragraph of the proof of the lemma). Then we have $w(S) \geq \eps_k^{-6} - 1$ by Claim~\ref{clm:saver}, and $h(S) \geq h(D_1) - h(D) - \eps_k^{-3}$ if $D_1 = D(D\cup S)$, also by Claim~\ref{clm:saver}. On the other hand, if $D_1$ is larger than $D(D\cup S)$ then
\[
h(D_1) - h(D) - \eps_k^{-3} \leq 2\gamma + 1 - \eps_k^{-3} \leq 1,
\]
and $h(S) \geq 1$ by the definition of the height of a droplet. Thus in either case $S$ satisfies the conditions of Definition~\ref{def:partition}~(2).

Finally, we note (once again) that $\Delta(D_1,D')$ occurs, by Observation~\ref{obs:Delta} (using the fact that $S$ being spanned by $(D' \setminus D) \cap A$ implies $S$ is also spanned by $(D_1 \setminus D) \cap A$, since $S \subset D_1 \subset D'$), and, for the induction on $\lfloor h \rfloor$, that $h(D') - h(D_1) \leq h(D') - h(D) - 1$. Thus, we are done by induction.
\end{proof}

From here, the proof of Lemma~\ref{lem:crossings} is no more than a calculation. First, we establish a bound for the probability of the existence of saver droplets.

\begin{lemma}\label{lem:calcsaver}
Let $\PP = (D_i)_{i=0}^m$ be a satisfied partition for $(D,D')$, where $D$ and $D'$ satisfy the conditions of Lemma~\ref{lem:crossings}. Let $w := w(D')$, and suppose that $\ih(k)$ holds. Then, for each $1\leq i\leq m-1$, the probability that $(D_i \setminus D_{i-1}) \cap A$ spans a saver droplet (that is, a droplet satisfying the conditions of Definition~\ref{def:partition}~(2)) is at most
\[
w^{O(1)} \cdot p^{(1-\eps_k)(1-\eps_k^2) y_i/2},
\]
where $y_i := h(D_i) - h(D_{i-1})$.
\end{lemma}

\begin{proof}
First we apply Lemma~\ref{lem:numdroplets} to count the number of choices for the saver $S_i$. Indeed, if the integer parts of the coordinates of the source of $S_i$ are fixed, and if $\lfloor w(S_i) \rfloor = a$, then the lemma implies that there are at most $a^{O(1)}$ distinct choices for $S_i$. Now, $S_i$ is spanned by $(D_i \setminus D_{i-1}) \cap A$, and therefore we have the inclusions $S_i \subset D_i \subset D'$, since $D_i$ and $D'$ are droplets. So the number of choices for the integer part of the source of $S_i$ is at most $w^2$. Hence, the total number of choices for $S_i$ is at most $w^{O(1)}$, independently of $h(S_i)$ and $y_i$. It only remains to show that the probability a given droplet $S_i$ satisfies the conditions of a saver droplet in Definition~\ref{def:partition}~(2) is at most
\begin{equation}\label{eq:saverprob}
p^{(1-\eps_k)(1-\eps_k^2) y_i/2}.
\end{equation}

Let $S_i$ be a droplet spanned by $(D_i \setminus D_{i-1}) \cap A$, such that the width and height of $S_i$ satisfy the conditions in~\eqref{eq:satisfied}, which we recall again here:
\begin{equation}\label{eq:satisfied2}
w(S_i) \geq \eps_k^{-6}-1 \qquad \text{and} \qquad h(S_i) \geq h(D_i)-h(D_{i-1}) - \eps_k^{-3}.
\end{equation}
 Note that it is possible that $h(S_i)$ is large: indeed it is possible that it is much larger than $y_i = h(D_i)-h(D_{i-1})$. If that is the case, then we may pass to a sub-droplet $S_i' \subset S_i$ as follows: if $h(S_i) \leq p^{-(2/3)^k}/(\log 1/p)$ then we set $S_i' := S_i$; otherwise, by Lemma~\ref{lem:AL}, we may choose a droplet $S_i' \subset S_i$ spanned by $(D_i\setminus D_{i-1})\cap A$ such that
\begin{equation}\label{eq:calcDstar}
%h(D_i)-h(D_{i-1})-\eps_k^{-3} \leq
\frac{p^{-(2/3)^k}}{2\log 1/p} \leq h(S_i') \leq \frac{p^{-(2/3)^k}}{\log 1/p}.
\end{equation}
In either case we have
\begin{equation}\label{eq:calcheight}
h(S_i') \geq h(D_i) - h(D_{i-1}) - \eps_k^{-3},
\end{equation}
because if $S_i' = S_i$ then this is just the second part of~\eqref{eq:satisfied2}, and if $S_i' \subsetneq S_i$ then
\[
h(S_i') \geq \frac{p^{-(2/3)^k}}{2\log 1/p} \geq h(D') - h(D) > h(D_i) - h(D_{i-1}) - \eps_k^{-3}.
\]
The probability $S_i'$ is spanned by $(D_i \setminus D_{i-1}) \cap A$ is at most the probability it is internally spanned, since if $S_i'$ is spanned by $(D_i \setminus D_{i-1}) \cap A$, then it is also spanned by $(D_i \setminus D_{i-1}) \cap A \cap S_i'$. Therefore, applying $\ih(k)$ (using the upper bound on $h(S_i')$ from~\eqref{eq:calcDstar}), we obtain
\[
%begin{equation}\label{eq:saverih}
\P_p\big(\ispan(S_i')\big) \leq p^{(1-\eps_k)h(S_i')/2}.
\]
%end{equation}
For droplets $S_i'$ with $h(S_i') \geq \eps_k^{-5}$, this bound will be sufficient. Indeed, in such cases we have $\eps_k^2 \cdot h(S_i') \geq \eps_k^{-3}$, and hence, by~\eqref{eq:calcheight},
\[
y_i \leq h(S_i') + \eps_k^{-3} \leq (1 + \eps_k^2) \cdot h(S_i') \leq \frac{h(S_i')}{1-\eps_k^2},
\]
so~\eqref{eq:saverprob} holds. For smaller saver droplets we need a better bound, because in these cases the error of $\eps_k^{-3}$ in the height bound in~\eqref{eq:satisfied2} is significant relative to $h(S_i)$.\footnote{We have returned to using the original saver droplet because if $h(S_i)$ is small then we do not need to pass to a sub-droplet $S_i'$.}  We obtain this by using the final condition of a saver droplet in Definition~\ref{def:partition}~(2): that if $h(S_i) < \eps_k^{-5}$ then the rightmost $\eps_k^{-6}-1$ columns of $S_i$ all contain an element of $(D_i \setminus D_{i-1}) \cap A$. The probability that this occurs is at most
\[
\big(ph(S_i)\big)^{\eps_k^{-6}-1} \leq p^{2\eps_k^{-5}},
\]
if $p$ is sufficiently small, since we are assuming $h(S_i) \leq \eps_k^{-5}$, and we have used the (easy) fact that $|\partial(S_i)| \leq h(S_i)$. The bound in~\eqref{eq:saverprob} now follows, since $h(S_i) < \eps_k^{-5}$ implies $y_i \leq 2\eps_k^{-5}$.
\end{proof}

We can now complete the proof of Lemma~\ref{lem:crossings}.

\begin{proof}[Proof of Lemma~\ref{lem:crossings}]
%Let us set $C=\eps_k^{-2}$ in Definitions~\ref{def:partition} and~\ref{def:partition}.
We shall show that the probability that $(D,D')$ admits a satisfied partition is at most the bound claimed in~\eqref{eq:crossings}; the lemma will then follow from Lemma~\ref{lem:partition}.

Thus, suppose $\PP=(D_i)_{i=0}^m$ is a satisfied partition for $(D,D')$, and let $w:=w(D')$. To start, we claim that for each $1\leq i\leq m-1$, the probability that $(D_i\setminus D_{i-1})\cap A$ contains a weakly connected $D_{i-1}$-rooted set $Y_i$ of size $\gamma = \lfloor \eps_k^{-3}/2 \rfloor$, given that $y_i:=h(D_i)-h(D_{i-1})\leq 2\gamma+2$, is at most
\begin{equation}\label{eq:calcstart}
2w \cdot \bigg(\frac{c_k}{f'(w)}\bigg)^{\eps_k^{-3}/2-1} \cdot p^{\eps_k^{-3}/2},
\end{equation}
where $c_k$ depends only on $\eps_k$. To see this, first note that each $y \in Y_i$ lies within vertical distance $2\gamma + 1$ of $D_{i-1}$, because $|Y_i|=\gamma$ and $Y_i$ is $D_{i-1}$-rooted. Then for each $y \in Y_i$, there are at most $O(\gamma) / f'(w)$ sites $y'$ such that $y \prec y'$ (here we have used that $f'$ is decreasing). Hence, when searching for elements of $Y_i$ greedily, there are only $c_k/f'(w)$ choices for each new site. Now if $p$ is sufficiently small then~\eqref{eq:calcstart} is at most
\begin{equation}\label{eq:calcgamma}
w \cdot \bigg(\frac{p}{f'(w)}\bigg)^{(1-\eps_k)y_i/2},
\end{equation}
since $2\gamma \leq 	y_i \leq 2\gamma + 2 \leq \eps_k^{-3} + 2$ (and $\eps_k$ being sufficiently small) implies $\eps_k^{-3}/2-1 \geq (1-\eps_k)y_i/2$, and since $2p \cdot c_k^{\eps_k^{-3}/2} \leq 1$, because $p$ is sufficiently small.

On the other hand, for each $1\leq i\leq m-1$, the probability that $(D_i\setminus D_{i-1})\cap A$ spans a saver droplet $S_i$ (that is, $S_i$ satisfies the conditions of Definition~\ref{def:partition}~(2)) is at most
\begin{equation}\label{eq:calcsaver}
w^{O(1)} \cdot p^{(1-\eps_k)(1-\eps_k^2)y_i/2},
\end{equation}
by Lemma~\ref{lem:calcsaver}, where as usual $y_i := h(D_i)-h(D_{i-1})$.

Next we combine the bound for weakly connected sets from~\eqref{eq:calcgamma} with the bound for saver droplets from~\eqref{eq:calcsaver}. If one defines for each $1 \leq i \leq m-1$ the event $\EE_i$ to be that $(D_i \setminus D_{i-1}) \cap A$ either contains a weakly connected set of size $\gamma$ or spans a saver droplet, then the events $\EE_i$ are independent as $i$ varies, even though the saver droplet spanned by $(D_i \setminus D_{i-1}) \cap A$ may not be fully contained in $D_i \setminus D_{i-1}$. This is because $\EE_i$ only depends on the intersection of $D_i \setminus D_{i-1}$ with $A$, and the sets $D_i \setminus D_{i-1}$ are disjoint for different values of $i$. Moreover, by~\eqref{eq:calcgamma} and~\eqref{eq:calcsaver},
\begin{equation}\label{eq:calceither}
\P_p(\EE_i) \leq w^{O(1)} \cdot \bigg(\frac{p}{f'(w)}\bigg)^{(1-\eps_k)(1-\eps_k^2)y_i/2}
\end{equation}
for each $1 \leq i \leq m-1$. Observe also that
\begin{equation}\label{eq:telescope}
\sum_{i=1}^{m-1} y_i = h(D_{m-1}) - h(D) = y - \big(h(D') - h(D_{m-1})\big) \geq y - 3\eps_k^{-3},
\end{equation}
by Definition~\ref{def:partition}. Noting that we always have
\begin{equation}\label{eq:mchoices}
m = O(\eps_k^3 y),
\end{equation}
since $h(D_i) - h(D_{i-1}) \geq 2\gamma \geq 2\eps_k^{-3}/3$ for each $1 \leq i \leq m-1$, it follows from~\eqref{eq:calceither} and~\eqref{eq:telescope} that
\begin{equation}\label{eq:calcprod}
\prod_{i=1}^{m-1} \P_p(\EE_i) \leq w^{O(\eps_k^3 y)} \cdot \left(\frac{p}{f'(w)}\right)^{(1-\eps_k)(1-\eps_k^2)(y-3\eps_k^{-3})/2}.
\end{equation}
In order to bound the probability that there is a satisfied partition for $(D,D')$, we take the union bound over the choices of $m$ and $D_1,\dots,D_{m-1}$. By Lemma~\ref{lem:numdroplets}, the number of choices for each $D_i$ is $w^{O(1)}$ (cf.~the proof of Lemma~\ref{lem:calcsaver}), so the total number of choices for $m$ and $D_1,\dots,D_{m-1}$ is at most $w^{O(\eps_k^3 y)}$, by~\eqref{eq:mchoices}. Hence, by~\eqref{eq:calcprod}, the probability there is a satisfied partition for $(D,D')$ is at most
\[
w^{O(\eps_k^3 y)} \cdot \left(\frac{p}{f'(w)}\right)^{(1-\eps_k)(1-\eps_k^2)(y-3\eps_k^{-3})/2}.
\]
We are given that $y \geq \eps_k^{-5}$, and therefore $y-3\eps_k^{-3} \geq y(1-3\eps_k^2)$. Hence, the preceeding probability is at most
\[
w^{O(\eps_k^3 y)} \cdot \left(\frac{p}{f'(w)}\right)^{(1-1.01\eps_k)y/2},
\]
and so as noted earlier, we are done by Lemma~\ref{lem:partition}.
\end{proof}

The deduction of Lemma~\ref{lem:delta} from Lemma~\ref{lem:crossings} proceeds as follows. Given $D\subset D'$, let $D_v$ be the minimal droplet such that $D\subset D_v\sqsubset D'$, and let $D_h$ be the maximal droplet such that $D\sqsubset D_h\subset D'$. Observe that
\begin{equation}\label{eq:vertorhoriz}
\Delta(D,D') \Rightarrow \Delta(D_v,D') \wedge \Delta(D_h,D').
\end{equation}
Now, either $h(D')-h(D_v)$ is large, in which case we bound the probability of the event $\Delta(D_v,D')$ using Lemma~\ref{lem:crossings}, or $w(D')-w(D_h)$ is large, in which case we bound the probability of the event $\Delta(D_h,D')$ directly by noting that every column of $D'\setminus D_h$ must intersect $A$ (see Figure~\ref{fig:delta}). We now give the details.

\begin{figure}
 \begin{minipage}{.45\textwidth}
  \centering
  \begin{tikzpicture}[>=latex,scale=0.7]
    \draw [domain=0:8] plot (\x,{1.4*ln(1+\x)});
    \draw [domain=0:8] plot (\x,{-1.4*ln(1+\x)});
    \draw [domain=-1:1,variable=\t] plot (8,{1.4*ln(9)*\t});
    \draw [domain=0:5,densely dashed] plot ({\x+2},{-0.2+1.4*ln(1+\x)});
    \draw [domain=0:5,densely dashed] plot ({\x+2},{-0.2-1.4*ln(1+\x)});
    \draw [domain=-1:1,variable=\t] plot (7,{1.4*ln(8)*\t});
%    \draw [domain=0:1,variable=\t] plot (7,{-0.2+1.4*ln(6) + (1.4*ln(8/6)+0.2)*\t});
%    \draw [domain=0:1,variable=\t] plot (7,{-0.2-1.4*ln(6) - (1.4*ln(8/6)-0.2)*\t});
    \draw [->] (7.2,2) -- (7.8,2);
    \draw [->] (7.2,0) -- (7.8,0);
    \draw [->] (7.2,-2) -- (7.8,-2);
    \node at (1.5,0.4) {$D_h$};
  \end{tikzpicture}
 \end{minipage}
 \begin{minipage}{.45\textwidth}
  \centering
  \begin{tikzpicture}[>=latex,scale=0.7]
    \draw [domain=0:8] plot (\x,{1.4*ln(1+\x)});
    \draw [domain=0:8] plot (\x,{-1.4*ln(1+\x)});
    \draw [domain=-1:1,variable=\t] plot (8,{1.4*ln(9)*\t});
    \draw [domain=0:5.5] plot ({\x+2.5},{-0.2+1.4*ln(1+\x)});
    \draw [domain=0:5.5] plot ({\x+2.5},{-0.2-1.4*ln(1+\x)});
    \draw [domain=-1:1,variable=\t,densely dashed] plot (7.5,{-0.2+1.4*ln(6)*\t});
%    \draw [domain=5:5.5,densely dotted,thick] plot ({\x+2.5},{-0.2+1.4*ln(1+\x)});
%    \draw [domain=5:5.5,densely dotted,thick] plot ({\x+2.5},{-0.2-1.4*ln(1+\x)});
    \draw [->] (4,1.2) -- (4,2.1);
    \draw [->] (6,2) -- (6,2.6);
    \draw [->] (4,-1.6) -- (4,-2.2);
    \draw [->] (6,-2.4) -- (6,-2.7);
%    \draw (7.75,0) -- (7,0) node [left] {$D_v$};
    \node at (6,0) {$D_v$};
  \end{tikzpicture}
 \end{minipage}
  \caption{The two cases of the proof of Lemma~\ref{lem:delta}. Both figures show the inner droplet $D$ and the outer droplet $D'$. On the left (Case~1), $w(D') - w(D_h)$ is large, the intermediate droplet shown is $D_h$, and in the proof we bound $\P_p\big(\Delta(D,D')\big)$ directly by noting that every column of $D' \setminus D_h$ must intersect $A$. On the right (Case~2), $h(D') - h(D_v)$ is large, the intermediate droplet shown is $D_v$, and in the proof we bound $\P_p\big(\Delta(D,D')\big)$ using Lemma~\ref{lem:crossings}.}
  \label{fig:delta}
\end{figure}
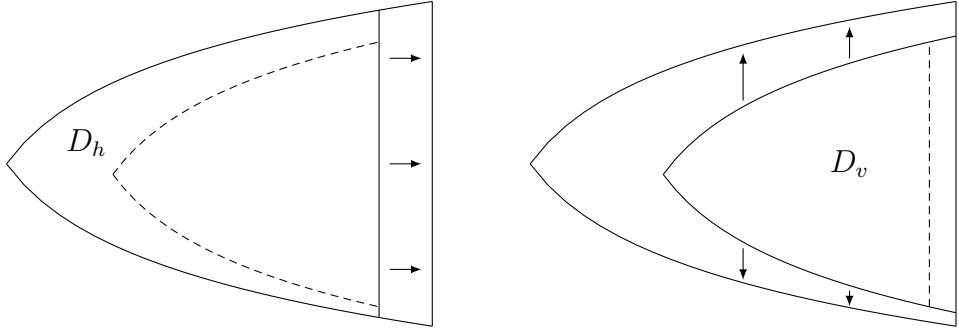

\begin{proof}[Proof of Lemma~\ref{lem:delta}]
Suppose that $\Delta(D,D')$ occurs, and let $D_v$ and $D_h$ be as above. By~\eqref{eq:vertorhoriz}, we have
\begin{equation}\label{eq:minupright}
\P_p\big(\Delta(D,D')\big) \leq \min \Big\{ \P_p\big(\Delta(D_v,D')\big), \, \P_p\big(\Delta(D_h,D')\big) \Big\}.
\end{equation}
To prove the lemma, we shall show that at least one term inside the minimum is at most the right-hand side of~\eqref{eq:delta}.

Let
\[
x_v := w(D_v)-w(D), \quad x_h := w(D_h)-w(D), \quad \text{and} \quad x:=w(D')-w(D),
\]
and note that $x_v+x_h=x$, because $w(D_v)-w(D)=w(D')-w(D_h)$. Note also that $x_v \in \Z$, since $D$ and $D_v$ have the same source. Let $y := h(D') - h(D)$, so that we have
\begin{equation}\label{eq:yoverx}
\frac{y}{x} = \frac{h(D') - h(D)}{w(D') - w(D)} = 2 \cdot \frac{f\big(w(D')\big) - f\big(w(D)\big)}{w(D') - w(D)}.
\end{equation}
Thus, using Observation~\ref{obs:f}~$(d)$ and the mean value theorem, we have that if $k\geq 1$ (and hence $f\big(w(D')\big) < h(D') \leq 1/4p$), then
\begin{equation}\label{eq:fprimeineq}
\frac{\log 1/p}{\eps^3} \cdot y \leq x \leq \frac{2\log 1/p}{\eps^3} \cdot y.
\end{equation}
%(The extra factor of $1/2$ comes from the two contributions to $y$: the contribution from the part of $D'\setminus D$ that lies above $D$, and the part that lies below $D$.)
\smallskip

\noindent {\bf Case 1.} First suppose that $x_v \geq \eps_k x / 50$. In this case we shall show that the probability $\P_p\big(\Delta(D_h,D')\big)$ of `crossing horizontally' is small: in fact we shall show that it is at most $p^y$, which is more than sufficient for the lemma.

The event $\Delta(D_h,D')$ implies that every column of $D'\setminus D_h$ is non-empty. If $k\geq 1$ then, since $x_v$ is an integer, it follows that
\[
\P_p\big(\Delta(D_h,D')\big) \leq \big( p \cdot h(D') \big)^{\eps_k x/50} \leq e^{-x},
%\big( 1-(1-p)^{h(D')} \big)^{\eps_k x/50} \leq \big( p \cdot h(D') \big)^{\eps_k x/50} \leq \big( p \cdot h(D') \big)^{\eps_k x/50}.
\]
where for the first inequality we have used the fact that $|\partial(D')| \leq h(D')$, and for the second inequality we have used the fact that $p\cdot h(D') = o(1)$ (which is true since $k\geq 1$). Combining this with~\eqref{eq:minupright} and~\eqref{eq:fprimeineq} we have
\begin{equation}\label{eq:delta1}
\P_p\big(\Delta(D,D')\big) \leq \exp\left( - \frac{\log 1/p}{\eps^3} \cdot y \right) \leq p^y.
\end{equation}
On the other hand, if $k=0$ then the probability that every column of $D'\setminus D_h$ is non-empty is at most
\[
\big( 1-(1-p)^{h(D')} \big)^{\eps_0 x/50} \leq \exp \left( - (1-p)^{h(D')} \cdot \frac{\eps^2}{50} \cdot x \right),
\]
where we have again used $|\partial(D')| \leq h(D')$, and we have also substituted $\eps_0 = \eps^2$. Thus, using the inequality $1-p \geq e^{-p-p^2}$ (since $p$ is sufficiently small), we have
\begin{equation}\label{eq:delta3}
\P_p\big(\Delta(D_h,D')\big) \leq \exp \left( - e^{-ph(D')} \cdot \frac{\eps^2}{100} \cdot x \right),
\end{equation}
since $e^{-p^2 h(D')} = 1 - o(1)$. Now observe that
\[
\frac{y}{x} \leq 2 \cdot f'\big(w(D)\big),
\]
by~\eqref{eq:yoverx}, the mean value theorem, and the fact that $f'$ is decreasing (Observation~\ref{obs:f}~$(b)$). Hence,
\[
x \geq \frac{\log 1/p}{\eps^3} \cdot e^{2pf(w(D))} \cdot y \geq \frac{\log 1/p}{2\eps^3} \cdot e^{ph(D)} \cdot y,
\]
by Observation~\ref{obs:f}~$(c)$, the definition of $h(D)$, and the fact that $e^{-p} \geq 1/2$. Inserting this into~\eqref{eq:delta3} and using the bound from~\eqref{eq:minupright} gives
\[
\P_p\big(\Delta(D,D')\big) \leq \exp\left( - e^{-p(h(D')-h(D))} \cdot \frac{\log 1/p}{200\eps} \cdot y \right) \leq \exp\left( - \frac{\log 1/p}{300\eps} \cdot y \right),
\]
since $p\big(h(D')-h(D)\big) = o(1)$. Thus, it follows that
\begin{equation}\label{eq:delta2}
\P_p\big(\Delta(D,D')\big) \leq p^y,
\end{equation}
since $\eps$ is sufficiently small. This together with~\eqref{eq:delta1} establishes the lemma in the case $x_v \geq \eps_k x / 50$. \smallskip

\noindent {\bf Case 2.} So suppose instead that $x_v \leq \eps_k x/50$. First we would like to show that $y':=h(D')-h(D_v)$ is not much smaller than $y$. To that end, note that
\begin{align*}
h(D_v)-h(D) &= 2f\big(w(D_v)\big) - 2f\big(w(D)\big) \\
&\leq 2f'\big( w(D) \big) \cdot x_v \\
&\leq 2f'\big( w(D) \big) \cdot \eps_k x / 50,
\end{align*}
by the mean value theorem and since $f'$ is decreasing. By a similar justification, and using~\eqref{eq:yoverx}, we have
\[
x \leq \frac{1}{2f'\big(w(D')\big)} \cdot y.
\]
Hence,
\[
h(D_v)-h(D) \leq \frac{f'\big(w(D)\big)}{f'\big(w(D')\big)} \cdot \frac{\eps_k}{50} \cdot y = e^{p(h(D')-h(D))} \cdot \frac{\eps_k}{50} \cdot y \leq \frac{\eps_k}{40} \cdot y,
\]
by Observation~\ref{obs:f}~$(c)$, the definition of the height of a droplet, and since $p \big( h(D') - h(D) \big) = o(1)$. Thus
\begin{equation}\label{eq:y'}
y'=h(D')-h(D_v) \geq (1-\eps_k/40)y.
\end{equation}

Note that the conditions of Lemma~\ref{lem:crossings} hold when applied to droplets $D_v$ and $D'$. Indeed, $D_v \sqsubset D'$ by construction; $h(D_v) \geq h(D) \geq \eps_k^{-5}$ by assumption;
\[
h(D') - h(D_v) \geq (1-\eps_k/40) \big( h(D') - h(D) \big) \geq (1-\eps_k/40) \cdot \eps_k^{-6} \geq \eps_k^{-5}
\]
by~\eqref{eq:y'} and assumption; and $h(D') - h(D_v) \leq p^{-(2/3)^k} \big(2\log 1/p\big)^{-1}$ again by assumption. Thus, applying Lemma~\ref{lem:crossings} gives
\[
%\P_p\big( \Delta(D_v,D') \big) \leq w(D')^{O(\eps_k^3 y')} \cdot \left( p \cdot \frac{\log 1/p}{\eps^2} \cdot e^{ph(D')} \right)^{(1-1.01\eps_k)y'/2}.
\P_p\big( \Delta(D_v,D') \big) \leq w(D')^{O(\eps_k^3 y')} \cdot \bigg( \frac{p}{f'\big(w(D')\big)} \bigg)^{(1-1.01\eps_k)y'/2}.
\]
We always have $h(D')\leq p^{-1}\log 1/p$ (regardless of $k$), which implies that
\[
w(D') = \frac{\log 1/p}{\eps^3 p} \Big( e^{p(h(D')-1)} - 1 \Big) \leq p^{-O(1)},
\]
by inverting the function $f$ and using the fact that $h(D') = 2f\big(w(D')\big) + 1$. Hence, also inserting the expression for $f'$ from Observation~\ref{obs:f}~$(c)$,
\[
\P_p\big( \Delta(D_v,D') \big) \leq p^{-O(\eps_k^3 y')} \cdot \left( p \cdot \frac{2\log 1/p}{\eps^3} \cdot e^{ph(D')} \right)^{(1-1.01\eps_k)y'/2}.
\]
Hence, using the (crude) bound
\[
\bigg(\frac{2\log 1/p}{\eps^3}\bigg)^{(1-1.01\eps_k)/2} \leq p^{-O(\eps_k^3)},
\]
we deduce that $\P_p\big( \Delta(D_v,D') \big)$ is at most
\begin{equation}\label{eq:horrible}
\exp\Bigg( - \Big(1-1.01\eps_k\Big) \bigg(\log\frac{1}{p} - ph(D') \bigg) \frac{y'}{2} + O\bigg( \eps_k^3 \log\frac{1}{p} \bigg) y' \Bigg).
\end{equation}

To deal with the final error term in~\eqref{eq:horrible}, we use the fact that $\log 1/p - ph(D') \geq \eps\log 1/p$. Together with~\eqref{eq:y'}, this gives us finally that
\[
\P_p\big( \Delta(D_v,D') \big) \leq \exp\Bigg( - \Big(1-1.1\eps_k\Big) \bigg(\log\frac{1}{p} - ph(D') \bigg) \frac{y}{2} \Bigg).
\]
We are now done by~\eqref{eq:minupright}.
\end{proof}

\section{Small droplets}\label{sec:small}

In this section we will bound the probability that a droplet of height at most $( p \log 1/p )^{-1}$ is internally spanned. Recall from Definition~\ref{def:ih} that, for each $k \geq 0$, we denote the following statement by $\ih(k)$:
\begin{itemize}
\item [] Let $D$ be a droplet of height at most $p^{-(2/3)^k} (\log 1/p)^{-1}$. Then
% IH(k) already has an eq ref: it's eq:ihbound (from section 2)
\[
\P_p\big(\ispan(D)\big) \leq p^{(1-\eps_k) h(D)/2},
\]
where $\eps_k = \eps^2 \cdot (3/4)^k$.
\end{itemize}
Our aim is to prove that $\ih(0)$ holds. This is an immediate consequence of the following two lemmas.

\begin{lemma}\label{lem:basecase}
$\ih(k)$ holds for all sufficiently large $k$.
\end{lemma}

\begin{lemma}\label{lem:indstep}
Let $k \geq 1$. Then $\ih(k) \Rightarrow \ih(k-1)$.
\end{lemma}

The proof of Lemma~\ref{lem:basecase} is easy, so the main task of this section will be to prove Lemma~\ref{lem:indstep}. We begin, however, with the more straightforward task. 

\begin{proof}[Proof of Lemma~\ref{lem:basecase}]
Let $k \in \N$ be sufficiently large, and let $D$ be a droplet with $h(D) \leq p^{-(2/3)^k} (\log 1/p)^{-1}$. By Lemma~\ref{lem:extremal}, if $D$ is internally spanned then
\[
|D\cap A| \geq \frac{h(D)+1}{2}.
\]
Noting that Observation~\ref{obs:f} implies that the volume of $D$ (rather crudely) satisfies
\[
|D| \leq (\log 1/p)^2 \cdot h(D)^2,
\]
it follows that
\[
\P_p\big(\ispan(D)\big) \le \binom{ |D| }{ \big(h(D)+1\big)/2 } p^{(h(D)+1)/2} = O\Big( h(D) \cdot p (\log 1/p)^2 \Big)^{(h(D)+1)/2}.
\]
But if $k$ is sufficiently large so that $\eps_k = \eps^2 \cdot (3/4)^k > (2/3)^k$, then $h(D) \cdot (\log 1/p)^2 \le p^{-(2/3)^k} \log 1/p \ll p^{-\eps_k}$, and hence this is at most $p^{(1-\eps_k)h(D)/2}$, as required.
\end{proof}

In order to prove Lemma~\ref{lem:indstep} we will use the method of hierarchies. In particular, we will use Lemmas~\ref{lem:boundoverH},~\ref{lem:numberofH} and~\ref{lem:delta}. 

In this section and the next, for the clearer display of expressions involving exponentials, we shall use the notation $\exp_p(x):=p^x$.

\begin{proof}[Proof of Lemma~\ref{lem:indstep}]
Let $k \ge 1$ and suppose that $\ih(k)$ holds. Let $D$ be a droplet with\footnote{If $h(D)$ is smaller than this, then the desired bound follows immediately from $\ih(k)$.} 
\[
p^{-(2/3)^k} (\log 1/p)^{-1} \le h(D) \le p^{-(2/3)^{k-1}} (\log 1/p)^{-1},
\]
and apply Lemma~\ref{lem:boundoverH} to $D$ with $t = p^{-(2/3)^k} / \big(4 \log 1/p\big)$. We obtain
\begin{equation}\label{eq:indbound}
\P_p\big(\ispan(D)\big) \leq \sum_{\hier \in \hier_D(t)} \bigg( \prod_{u \in L(\hier)} \P_p\big(\ispan(D_u)\big) \bigg)\bigg( \prod_{u \to v} \P_p\big(\Delta(D_v,D_u)\big) \bigg).
\end{equation}
To deduce the desired bound from~\eqref{eq:indbound}, we shall use $\ih(k)$ and Lemmas~\ref{lem:numberofH} and~\ref{lem:delta}. 

Let $\hier \in \hier_D(t)$, and note first that $t \leq h(D_u) \leq 2t = p^{-(2/3)^k} / \big(2\log 1/p\big)$ for every $u \in L(\hier)$, so by $\ih(k)$ we have
\begin{equation}\label{eq:indseed}
\P_p\big(\ispan(D_u)\big) \leq p^{(1-\eps_k)h(D_u)/2} \leq p^{t/3}.
\end{equation}
Next, note that if $u\to v$ then $h(D_u) - h(D_v) \le 2t = p^{-(2/3)^k} / \big(2\log 1/p\big)$. 
%Need to sort out what happens if this difference is in fact less than $\eps_k^{-4}$. 
If we also have $h(D_u)-h(D_v)\geq \eps_k^{-6}$, then by Lemma~\ref{lem:delta} we have
\begin{equation}\label{eq:inddelta}
\P_p\big(\Delta(D_v,D_u)\big) \leq \exp_p \bigg( \frac{(1-1.1\eps_k)(1-\eps_k^2)}{2} \Big( h(D_u)-h(D_v) \Big) \bigg),
\end{equation}
since $ph(D_u)\leq (\log 1/p)^{-1} \leq \eps_k^2\cdot\log 1/p$. Therefore we have
\begin{equation}\label{eq:indprod}
\prod_{u\to v} \P_p\big(\Delta(D_v,D_u)\big) \leq \exp_p \Bigg( \frac{1-\eps_k'}{2} \bigg( \sum_{u\to v} \Big( h(D_u)-h(D_v) \Big) - v(\hier) \cdot \eps_k^{-6} \bigg) \Bigg),
\end{equation}
where $1-\eps_k':=(1-1.1\eps_k)(1-\eps_k^2)$, and the second term in the exponential takes account of the fact that~\eqref{eq:inddelta} requires $h(D_u)-h(D_v) \geq \eps_k^{-6}$.

With foresight, let us split the sum in~\eqref{eq:indbound} into two parts, depending on the number of seeds in $\hier$. To that end, set $\ell_0 := t \cdot (\log 1/p)^{-1}$, and let
\[
\hier^{(1)} = \big\{ \hier \in \hier_D(t) : \ell(\hier) \le \ell_0 \big\} \quad \textup{and} \quad \hier^{(2)} = \hier_D(t) \setminus \hier^{(1)}.
\]
Bounding the sum over $\hier \in \hier^{(2)}$ is easy: by Lemma~\ref{lem:numberofH} and~\eqref{eq:indseed} we have  
\[
\sum_{\hier \in \hier^{(2)}} \prod_{u \in L(\hier)} \P_p\big(\ispan(D_u)\big) \le \sum_{\ell\geq\ell_0} \exp_p\Big( \ell \cdot t / 3 - O\big( \ell \cdot h(D) / t \big) \Big) < p^{h(D)},
\]
where the last inequality holds since $h(D) / t \ll t$ and $\ell_0 \cdot t \gg h(D)$.

%\footnote{Obtaining the bound $h(D)/t\ll t$ is the only point in the proof of Theorem~\ref{thm:Duarte} where we use the `$2/3$' in the exponent in the bound on $h(D)$ in the induction hypothesis: were it not for this, we could use $1/2$ instead.}

Thus, combining~\eqref{eq:indbound} with~\eqref{eq:indseed} and~\eqref{eq:indprod}, and noting that $\eps_k' > \eps_k$, it will suffice to bound
\begin{equation}\label{eq:indbound2}
\sum_{\hier \in \hier^{(1)}} \exp_p \Bigg( \frac{1-\eps_k'}{2} \bigg( \sum_{u \in L(\hier)} h(D_u) + \sum_{u\to v} \big( h(D_u)-h(D_v) \big) - v(\hier)\cdot \eps_k^{-6} \bigg) \Bigg).
\end{equation}
To do so, let $\hier \in \hier^{(1)}$, and recall that
\begin{equation}\label{eq:smallb1}
\sum_{u \in L(\hier)} h(D_u) + \sum_{u\to v} \big( h(D_u)-h(D_v) \big) \geq h(D) - v(\hier),
\end{equation}
by Lemma~\ref{lem:sumofheights}, and that 
\[
v(\hier) = O\bigg( \frac{\ell \cdot h(D)}{t} \bigg) = o\big( h(D) \big),
\]
by Lemma~\ref{lem:numberofH}, and since $\ell \leq \ell_0 = o(t)$. Thus, using Lemma~\ref{lem:numberofH} to bound $|\hier^{(1)}|$, it follows that
\[
\P_p\big(\ispan(D_u)\big) \leq \exp_p \Bigg( \bigg(\frac{1-\eps_k'}{2}\bigg)h(D) - o\big(h(D)\big) \Bigg) + p^{h(D)},
\]
where the $o\big(h(D)\big)$ in the exponent counts the size of $\hier^{(1)}$ and also the error of $O\big(v(\hier)\big)$. Since $\eps_{k-1} = (4/3)\cdot\eps_k$, this is at most $p^{(1 - \eps_{k-1}) h(D) / 2}$, as required.
\end{proof}

\section{Large droplets, and the proof of Theorem~\ref{thm:Duarte}}\label{sec:large}

In this section we shall prove Proposition~\ref{prop:lower}, and deduce Theorem~\ref{thm:Duarte}. The spirit of this section is similar to that of the previous section, in that we are proving an upper bound on the probability that a droplet is internally spanned assuming that we already have a corresponding bound for smaller droplets. This time, however, the larger droplets will be critical droplets and the smaller droplets will be those which we can bound using $\ih(0)$. Another important difference is that, as we reach the critical size, we gain an additional factor of $1/2$ in the exponent in the bound for $\P_p\big(\ispan(D)\big)$. Indeed, as one can see below in Proposition~\ref{prop:largedroplets}, the factor of $1/2$ decreases to $1/4$ linearly in the height of the droplet as the droplet reaches the critical size.

Given a droplet $D$, let
\begin{equation}\label{eq:hstar}
h^*(D):= \frac{p}{\log 1/p} \cdot h(D)
\end{equation}
denote the renormalized height of $D$. Proposition~\ref{prop:lower} is an immediate consequence of the following bound. %(Recall that $\eps > 0$ was fixed earlier.) 

\begin{prop}\label{prop:largedroplets}
For every $\eps > 0$, there exists $p_0(\eps) > 0$ such that the following holds. If $0 < p \leq p_0(\eps)$ and $D$ is a droplet with $h^*(D) \leq 1 - \eps$, then
\begin{equation}\label{eq:largedroplets}
\P_p\big(\ispan(D)\big) \leq \exp_p \Bigg( \left( \frac{2 - h^*(D)}{4} - \eps \right) h(D) \Bigg).
\end{equation}
\end{prop}

We will prove Proposition~\ref{prop:largedroplets} by taking a union bound over good and satisfied hierarchies for $D$. In order to do so, we will need one additional lemma, which bounds the product of the probabilities of the `sideways steps' of such a hierarchy. Define the \emph{pod height}\footnote{This terminology is a reference to the `pod' droplets first introduced in~\cite{Hol}. In our setting it will be more convenient to work with the pod height function directly.} of a hierarchy $\hier$ for a droplet $D$ to be
\begin{equation}\label{eq:poddef}
h(\hier) := \min\bigg\{ h(D), \sum_{u\in L(\hier)} h(D_u) \bigg\},
\end{equation}
and let $h^*(\hier) := p(\log 1/p)^{-1} \cdot h(\hier)$ be the renormalized pod height. Let us write $\ell(\hier)$ for $\big| L(\hier) \big|$, and set 
\[
t := \frac{1}{4p\log 1/p}.
\]
Finally, we will need a function $\mu$, defined by
\begin{equation}\label{eq:mu}
\mu(\hier) := \frac{2 - h^*(D) - h^*(\hier)}{4}.
\end{equation}
Note that if $h^*(\hier) \le h^*(D) \le 1 - \eps$, which will always be the case in this section, then $\mu(\hier) \ge \eps/2$. The following bound is a variant of~\cite[Lemma~38]{Hol}. We remark that such `pod lemmas' have since become a standard tool in the area; see e.g.~\cite{Hol,BBDM,DH,Msharp}. The proof follows (as usual) by adapting the argument of~\cite{Hol}, but since in our setting there are some slightly subtle complications to deal with, we will give the details in full.

\begin{lemma}\label{lem:pod}
Let $D$ be a droplet with $h^*(D) \le 1 - \eps$, and let $\hier$ be a $t$-good and satisfied hierarchy for $D$. Then
\begin{equation}\label{eq:sidewaysprod}
\prod_{u \to v} \P_p\big( \Delta(D_v,D_u) \big) \leq \exp_p \Big( \big(\mu(\hier)-2\eps^2\big)\big(h(D)-h(\hier)\big) - \eps^{-6}\big(3\ell(\hier)-2\big) \Big).
\end{equation}
\end{lemma}

We will use the following easy algebraic facts in the proof of Lemma~\ref{lem:pod}.

\begin{obs}\label{obs:pod1}
Let $a,a',s,s',\delta\in\R$. If $s' \leq s \leq 1 - 2\delta$, $a \geq a'$, and $2\delta(1+a) \geq a-a'$, then
\[
\left(\frac{2-a'-s'}{4}-\delta\right)(a'-s') + (1-\delta)\left(\frac{1-a}{2}\right)(a-a') \geq \left(\frac{2-a-s}{4}-\delta\right)(a-s).
\]
\end{obs}

\begin{proof}
The condition $s' \leq s \leq 1 - 2\delta$ implies that the left-hand side is decreasing in $s'$, so we may assume that $s=s'$. Then the claimed inequality is just a rearrangement of $2\delta(1+a)(a-a') \geq (a-a')^2$.
\end{proof}

\begin{obs}\label{obs:pod2}
Let $\delta,a,a_1,a_2,s,s_1,s_2 \in \R$. If $a,s \le 1 - 2\delta$, $a \leq a_1 + a_2$, $s \ge  s_1 + s_2$, and $a_1 a_2 \geq s_1 s_2$, then
\begin{multline*}
\left(\frac{2-a_1-s_1}{4} - \delta\right)(a_1 - s_1) + \left(\frac{2 - a_2 - s_2}{4} - \delta \right)( a_2 - s_2 ) \\ 
\geq \left( \frac{2 - a - s}{4} - \delta \right) ( a - s ).
\end{multline*}
\end{obs}

\begin{proof}
The right-hand side is increasing in $a$ and decreasing in $s$, since $a,s \le 1 - 2\delta$, so we may assume that $a =  a_1 + a_2$ and $s = s_1 + s_2$, in which case the inequality is equivalent to $a_1a_2 \ge s_1s_2$. 
\end{proof}

\begin{proof}[Proof of Lemma~\ref{lem:pod}]
The proof is by induction on $m := |V(G_\hier)|$. Note that the inequality holds trivially if $h(\hier) = h(D)$, since the right-hand side is at least 1, and that $h(\hier) = h(D)$ if $m = 1$. So let $m \ge 2$, and suppose that $h(\hier) < h(D)$ (so that in fact $h(\hier) = \sum_{u\in L(\hier)} h(D_u)$), and that the lemma holds for all hierarchies with at most $m-1$ vertices. We shall divide the induction step into two cases according to whether or not the first step of the hierarchy is a reasonably large sideways step.

\medskip
\noindent {\bf Case 1:} $N_{G_\hier}^\to(\root) = \{ w \}$ and $h(D) - h(D_w) \geq \eps^{-6}$.
\medskip

In this case the desired bound follows from Lemma~\ref{lem:delta}, $\ih(0)$ and the induction hypothesis on $m$, using Observation~\ref{obs:pod1}. To see this, set $D' = D_w$ and write $\hier'$ for the hierarchy obtained from $\hier$ by removing the vertex (and droplet) corresponding to $\root$, and adding a new root at $w$. Then, trivially,
\begin{equation}\label{eq:pod1}
\prod_{\substack{u \to v \\ u,v\in \hier}} \P_p\big( \Delta(D_v,D_u) \big) = \P_p\big( \Delta(D',D) \big) \prod_{\substack{u \to v \\ u,v\in \hier'}} \P_p\big( \Delta(D_v,D_u) \big).
\end{equation}
Now, observe that $\hier'$ is a $t$-good and satisfied hierarchy for $D'$. Thus, by the induction hypothesis on $m$, we have
\begin{equation}\label{eq:pod3}
\prod_{\substack{u \to v \\ u,v\in \hier'}} \P_p\big( \Delta(D_v,D_u) \big) \leq \exp_p \Big( \big( \mu(\hier') - 2\eps^2 \big) \big( h(D') - h(\hier') \big) - \eps^{-6} \big( 3\ell(\hier) - 2 \big) \Big), % \notag \\
%&\leq \exp_p \Big( \big( \mu(\hier') - 2\eps^2 \big) \big( h(D') - h(\hier) \big) - \eps^{-5} \big( \ell(\hier) - 2 \big) \Big),
\end{equation}
where we have replaced $\ell(\hier')$ by $\ell(\hier)$ since $L(\hier') = L(\hier)$.
%, implying that $\ell(\hier') = \ell(\hier)$, $h(\hier') \le h(\hier)$, and $\mu(\hier') \ge \mu(\hier) \ge 2\eps^2$.
Now, since $\ih(0)$ holds (by Lemmas~\ref{lem:basecase} and~\ref{lem:indstep}), and we have the bounds $\eps^{-6} \le h(D) - h(D') \le 2t$ and $h^*(D) \leq 1 - \eps$, we may apply Lemma~\ref{lem:delta} to give
\begin{equation}\label{eq:pod2}
\P_p\big( \Delta(D',D) \big) \leq \exp_p \bigg( (1-2\eps^2)\left(\frac{1-h^*(D)}{2}\right)\big(h(D)-h(D')\big) \bigg).
\end{equation}
Combining~\eqref{eq:pod3} and~\eqref{eq:pod2} with~\eqref{eq:pod1}, it follows that it is sufficient to show
\begin{multline}\label{eq:pod3.5}
\big(\mu(\hier')-2\eps^2\big)\big(h(D')-h(\hier')\big) + (1-2\eps^2)\left(\frac{1-h^*(D)}{2} \right)\big(h(D)-h(D')\big) \\
\geq \big(\mu(\hier)-2\eps^2\big)\big(h(D)-h(\hier)\big).
\end{multline}
We would like to apply Observation~\ref{obs:pod1} with $a = h^*(D)$, $a' = h^*(D')$, $s = h^*(\hier)$, $s' = h^*(\hier')$, and $\delta=2\eps^2$. If the conditions of the observation are satisfied, then we will be done, since on multiplying through by $p^{-1}\log 1/p$, the conclusion of the observation (with these parameters) is equivalent to~\eqref{eq:pod3.5}. For the conditions, we have: $h^*(D) \geq h^*(D')$ by assumption; $h^*(\hier) \geq h^*(\hier')$ from the previous inequality and since $L(\hier) = L(\hier')$; $h^*(\hier) \leq 1 - 4\eps^2$ since $h^*(\hier) \leq h^*(D)$ and $h^*(D) \leq 1-\eps$; and finally, 
\[
4\eps^2 \big( 1 + h^*(D) \big) \geq h^*(D) - h^*(D')
\]
since $h^*(D)-h^*(D') \leq (\log 1/p)^{-2} \ll 1$, by our choice of $t$. This completes the proof of the lemma in Case 1.

\medskip
\noindent {\bf Case 2:} $N_{G_\hier}^\to(\root)=\{w\}$ and $h(D) - h(D_w) < \eps^{-6}$.
\medskip

By the definition of a $t$-good hierarchy, there are two ways that we could have $h(D) - h(D') < \eps^{-6}$. One is that $w$ is a split vertex (which is why we have not considered separately the case in which $\root$ is a split vertex; see below), and the other is that $w$ is a leaf. If $w$ is a leaf then~\eqref{eq:sidewaysprod} trivially holds, since then $\root$ and $w$ are the only vertices in $\hier$, and the expression inside the exponent in~\eqref{eq:sidewaysprod} is at most
\[
\big( h(D) - h(D_w) \big) / 2 - \eps^{-6} < 0,
\]
so the right-hand side of~\eqref{eq:sidewaysprod} is greater than 1.

Thus we may assume that $w$ is a split vertex. (As mentioned above, we have not considered the case in which $\root$ is a split vertex. However, this
%  The alert reader will have noticed that we have not considered the case in which $\root$ is a split vertex. However, by the definition of a $t$-good hierarchy, the condition $h(D) - h(D_w) < \eps^{-5}$ implies that $w$ is a split vertex or a leaf, and therefore
case is covered by the calculation below, as long as we allow $h(D) - h(D_w)$ to be zero.\footnote{In this case, set $w = \root$ and $\hier' = \hier$ in the definitions in the next paragraph.}) We shall show that the desired bound follows from the induction hypothesis on $m$ directly, using Observation~\ref{obs:pod2}.

Indeed, set $D' = D_w$ and write $\hier'$ for the hierarchy obtained from $\hier$ by removing the vertex (and droplet) corresponding to $\root$, and adding a new root at $w$. Moreover, let $N_{G_\hier}^\to(w)=\{v_1,v_2\}$, and, for each $i \in \{1,2\}$, set $D_i = D_{v_i}$ and let $\hier_i$ be the part of $\hier'$ below and including $v_i$. Note that
\begin{equation}\label{eq:pod4}
\prod_{\substack{u \to v \\ u,v\in \hier'}} \P_p\big( \Delta(D_v,D_u) \big) = \prod_{\substack{u \to v \\ u,v\in \hier_1}} \P_p\big( \Delta(D_v,D_u) \big) \prod_{\substack{u \to v \\ u,v\in \hier_2}} \P_p\big( \Delta(D_v,D_u) \big) .
\end{equation}

Now, observe that $\hier_1$ and $\hier_2$ are $t$-good and satisfied hierarchies for $D_1$ and $D_2$. Therefore, by the induction hypothesis, we have
\begin{equation}\label{eq:pod5}
\prod_{\substack{u \to v \\ u,v\in \hier_i}} \P_p\big( \Delta(D_v,D_u) \big) \leq \exp_p \Big( \big(\mu(\hier_i)-2\eps^2\big)\big(h(D_i)-h(\hier_i)\big) - \eps^{-6}\big(3\ell(\hier_i)-2\big) \Big),
\end{equation}
for each $i \in \{1,2\}$. Moreover, we have
\begin{equation}\label{eq:sineq}
h(\hier) = \sum_{u\in L(\hier)} h(D_u) \ge h(\hier_1) + h(\hier_2)
% \quad \text{and} \quad
\end{equation}
since we assumed $h(\hier) < h(D)$, and we also have
\begin{equation}\label{eq:aineq}
h(D) \leq h(D') + \eps^{-6} \leq h(D_1) + h(D_2) + 1 + \eps^{-6}
\end{equation}
by Lemma~\ref{lem:subadd}. 

Next we shall apply Observation~\ref{obs:pod2} with $a = h^*(D) - (1 + \eps^{-6}) p ( \log 1/p )^{-1}$, $s = h^*(\hier)$, $a_i = h^*(D_i)$ and $s_i = h^*(\hier_i)$ for $i \in \{1,2\}$, and $\delta=2\eps^2$. This is permissible since we have $a \leq a_1 + a_2$ by~\eqref{eq:aineq}, $s \geq s_1 + s_2$ by~\eqref{eq:sineq}, $a_1 a_2 \geq s_1 s_2$ since $a_i \geq s_i$ for $i \in \{1,2\}$ by the definition of $h(\hier_i)$, and finally $a,s \leq 1 - 2\delta$ since $s \leq a + \eps^2$ (say) and $a \leq 1 - \eps$ by the assumption of the lemma. Applying Observation~\ref{obs:pod2} and multiplying through by $p^{-1} \log 1/p$ gives
\begin{multline*}
%\label{eq:combinepods}
\big( \mu(\hier_1) - 2\eps^2 \big) \big( h(D_1) - h(\hier_1) \big) + \big( \mu(\hier_2) - 2\eps^2 \big) \big( h(D_2) - h(\hier_2) \big) \\
\geq \big( \mu(\hier) - 2\eps^2 - (1 + \eps^{-6})p(\log 1/p)^{-1} \big) \big( h(D) - h(\hier) - (1 + \eps^{-6}) \big),
\end{multline*}
%The first bracket is positive since $h^*(\hier) \leq h^*(D) \leq 1 - \eps$, and therefore $\mu(\hier) \geq \eps/2$. Hence
After rearranging, the right-hand side is at least
\[
\big( \mu(\hier) - 2\eps^2 \big) \big( h(D) - h(\hier) \big) - (1 + \eps^{-6}) \big( \mu(\hier) + h^*(D) \big),
\]
so all together we have
\begin{multline}\label{eq:combinepods}
\big( \mu(\hier_1) - 2\eps^2 \big) \big( h(D_1) - h(\hier_1) \big) + \big( \mu(\hier_2) - 2\eps^2 \big) \big( h(D_2) - h(\hier_2) \big) \\
\geq \big( \mu(\hier) - 2\eps^2 \big) \big( h(D) - h(\hier) \big) - 2 \eps^{-6},
\end{multline}
since $\mu(\hier) \leq 1/2$ and $h^*(D) \leq 1$.

Returning to the probability we wish to bound, after combining~\eqref{eq:pod4} and~\eqref{eq:pod5} with~\eqref{eq:combinepods} we have that the left-hand side of~\eqref{eq:pod4} is at most
\[
\exp_p \Big( \big( \mu(\hier) - 2\eps^2 \big) \big( h(D) - h(\hier) \big) - 2\eps^{-6} - \eps^{-6} \big( 3\ell(\hier_1) + 3\ell(\hier_2) - 4 \big) \Big).
\]
The proof of the lemma is now complete, since $\ell(\hier) = \ell(\hier_1) + \ell(\hier_2)$, and we can bound $\P_p\big(\Delta(D',D)\big)$ trivially by $1$, which gives
\[
\prod_{\substack{u \to v \\ u,v \in \hier}} \P_p\big( \Delta(D_v,D_u) \big) \leq \exp_p \Big( \big(\mu(\hier)-2\eps^2\big)\big(h(D)-h(\hier)\big) - \eps^{-6}\big(3\ell(\hier)-2\big) \Big),
\]
as desired.
%Thus, combining~\eqref{eq:pod4} and~\eqref{eq:pod5}, noting that $\ell(\hier) = \ell(\hier_1) + \ell(\hier_2)$, and applying Observation~\ref{obs:pod2} with $a = h^*(D') - p ( \log 1/p )^{-1}$, $s = h^*(\hier)$, $a_i = h^*(D_i)$ and $s_i = h^*(\hier_i)$ for $i\in\{1,2\}$, and $\delta=2\eps^2$, we obtain
%\[
%\prod_{\substack{u \to v \\ u,v\in \hier'}} \P_p\big( \Delta(D_v,D_u) \big) \leq \exp_p \Big( \big(\mu(\hier')-2\eps^2\big)\big(h(D')-h(\hier)-1\big) - \eps^{-5}\big(\ell(\hier)-2\big) - 1 \Big)
%\]
%since $h^*(\hier_i) \le h^*(D_i)$ for $i \in \{1,2\}$ by definition.
%In order to complete the proof of the lemma, it only remains to verify that
%\begin{equation}\label{eq:pod7}
%\big(\mu(\hier)-2\eps^2\big)\big(h(D)-h(\hier)\big) - \big(\mu(\hier')-2\eps^2\big)\big(h(D')-h(\hier) - 1\big) \leq \eps^{-5} - 1.
%\end{equation}
%After rearranging, the left-hand side is equal to
%\[
%\big( h(\hier) - h(D) \big) \big( \mu(\hier') - \mu(\hier) \big) + \big(\mu(\hier')-2\eps^2\big) \big( h(D) -  h(D') + 1 \big),
%\]
%and since $\mu(\hier) \le \mu(\hier') \le 1/2$, $h(\hier) \le h(D)$ and $h(D) -  h(D') \le \eps^{-5}$, it follows that this is at most $(\eps^{-5} + 1)/2 < \eps^{-5}$, as required. This completes the proof in Case~2, and the lemma follows.
\end{proof}

We now have all the tools we need in order to prove Proposition~\ref{prop:largedroplets}. %In particular, we will use the induction hypothesis $\ih(0)$ (which follows from Lemmas~\ref{lem:basecase} and~\ref{lem:indstep}) and Lemmas~\ref{lem:boundoverH},~\ref{lem:numberofH} and~\ref{lem:pod}.

\begin{proof}[Proof of Proposition~\ref{prop:largedroplets}]
Let $D$ be a droplet such that $h^*(D) \le 1 - \eps$, set $t = (4p\log 1/p)^{-1}$, and note that we may assume that $h(D) \ge t$, since otherwise the lemma follows immediately from $\ih(0)$. Applying Lemma~\ref{lem:boundoverH} to $D$, we obtain
\begin{equation}\label{eq:largebound}
\P_p\big(\ispan(D)\big) \leq \sum_{\hier \in \hier_D(t)} \bigg( \prod_{u \in L(\hier)} \P_p\big(\ispan(D_u)\big) \bigg)\bigg( \prod_{u \to v} \P_p\big(\Delta(D_v,D_u)\big) \bigg).
\end{equation}
In order to deduce Proposition~\ref{prop:largedroplets} from~\eqref{eq:largebound}, we shall use $\ih(0)$ and Lemmas~\ref{lem:numberofH} and~\ref{lem:pod}.  

Let $\hier \in \hier_D(t)$, and note that $h(D_u) \leq 2t = (2p\log 1/p)^{-1}$ for every $u \in L(\hier)$. Thus, by $\ih(0)$ (which follows from Lemmas~\ref{lem:basecase} and~\ref{lem:indstep}), we have
\begin{equation}\label{eq:largeseed}
\prod_{u\in L(\hier)} \P_p\big(\ispan(D_u)\big) \leq \prod_{u\in L(\hier)} p^{(1-\eps^2)h(D_u)/2} \leq p^{(1-\eps^2)h(\hier)/2}.
\end{equation}
%for every such seed $u$.
Also, by Lemma~\ref{lem:pod}, we have
\begin{equation}\label{eq:largedelta}
\prod_{u \to v} \P_p\big( \Delta(D_v,D_u) \big) \leq \exp_p \Big( \big(\mu(\hier)-2\eps^2\big)\big(h(D)-h(\hier)\big) - \eps^{-6}\big(3\ell(\hier)-2\big) \Big).
\end{equation}
As in the proof of Lemma~\ref{lem:indstep}, we split the sum in~\eqref{eq:largebound} into two parts, depending on the number of seeds in $\hier$. Thus, let us set 
\[
\hier^{(1)} = \big\{ \hier \in \hier_D(t) : \ell(\hier) \le p^{-1/2} \big\} \quad \textup{and} \quad \hier^{(2)} = \hier_D(t) \setminus \hier^{(1)}.
\]
As before, bounding the sum over $\hier \in \hier^{(2)}$ is easy: by Lemma~\ref{lem:numberofH} and~\eqref{eq:largeseed} we have  
%\begin{equation}\label{eq:largebound:Htwo}
\begin{equation}\label{eq:H2bound}
\sum_{\hier \in \hier^{(2)}} \prod_{u \in L(\hier)} \P_p\big(\ispan(D_u)\big) \le \sum_{\ell \ge p^{-1/2}} \exp_p\Big(  \ell \cdot t / 3 - O\big( \ell \cdot h(D) / t \big) \Big) < e^{-p^{-5/4}},
\end{equation}
%\end{equation}
where the last inequality holds since $h(D) / t =O\big((\log 1/p)^2\big)$ and $t > p^{-3/4}$. 

For the sum over $\hier \in \hier^{(1)}$, we insert the bounds from~\eqref{eq:largeseed} and~\eqref{eq:largedelta} into~\eqref{eq:largebound} to obtain
\begin{multline}\label{eq:H1bound}
\sum_{\hier \in \hier^{(1)}} \bigg( \prod_{u \in L(\hier)} \P_p\big(\ispan(D_u)\big) \bigg)\bigg( \prod_{u \to v} \P_p\big(\Delta(D_v,D_u)\big) \bigg) \\
\leq \sum_{\hier \in \hier^{(1)}} \exp_p \bigg( \big(\mu(\hier) - 2\eps^2\big) \big(h(D) - h(\hier)\big) + \bigg(\frac{1-\eps^2}{2}\bigg) h(\hier) - \eps^{-6}\big(3\ell(\hier)-2\big) \bigg).
\end{multline}
Observe that by rearranging the terms and noting that $h(D) h^*(\hier) = h(\hier) h^*(D)$, we have
\begin{multline*}
\left( \frac{2-h^*(D)-h^*(\hier)}{4} - 2\eps^2 \right) \big( h(D) - h(\hier) \big) + \left( \frac{1-\eps^2}{2} \right) h(\hier) \\
\geq \left( \frac{2-h^*(D)}{4} - 2\eps^2 \right) h(D),
\end{multline*}
and therefore~\eqref{eq:H1bound} is at most
\begin{equation}\label{eq:H1bound2}
\sum_{\hier \in \hier^{(1)}} \exp_p \Bigg( \bigg(\frac{2 - h^*(D)}{4} - 2\eps^2\bigg) h(D) - \eps^{-6}\big(3\ell(\hier)-2\big) \Bigg).
\end{equation}
By Lemma~\ref{lem:numberofH} and the bounds $\ell(\hier) \leq p^{-1/2}$ and $h(D)/t \leq (\log 1/p)^2$, we have 
\begin{equation}\label{eq:H1bound3}
|\hier^{(1)}| \leq p^{-1/2} \cdot \exp\bigg( O\bigg( \frac{h(D) \log 1/p }{t \sqrt{p}} \bigg) \bigg) < e^t.
\end{equation}
Finally, combining~\eqref{eq:H1bound2} with~\eqref{eq:H1bound3} and the bounds $h(D) \ge t \gg p^{-1/2} \ge \ell(\hier)$, which hold for every $\hier \in \hier^{(1)}$, and adding~\eqref{eq:H2bound}, it follows that
\[
\P_p\big( \ispan(D) \big) \leq \exp_p \Bigg( \bigg( \frac{2-h^*(D)}{4} - \eps \bigg) h(D) \Bigg),
\]
as required.
\end{proof}

We are finally ready to complete the proof of Theorem~\ref{thm:Duarte}. 

\begin{proof}[Proof of Theorem~\ref{thm:Duarte}]
The upper bound was proved in Section~\ref{sec:upper}, so fix $\lambda < 1/8$, and set
\[
p = \frac{\lambda (\log\log n)^2}{\log n}.
\]
We will prove that with high probability a $p$-random subset $A \subset \Z_n^2$ does not percolate.

Indeed, if $A$ percolates then, by Lemma~\ref{lem:critical:droplet}, there exists a pair $(D_1,D_2)$ of disjointly internally spanned droplets such that
\[
\max\big\{ h(D_1), h(D_2) \big\} \leq \frac{1-\eps}{p} \log \frac{1}{p} \quad \text{and} \quad h(D_1) + h(D_2) \geq \frac{1-\eps}{p} \log \frac{1}{p} - 1,
\]
and $d(D_1,D_2) \le 2$. By Lemma~\ref{lem:numdroplets}, there are at most $n^2 \cdot p^{-O(1)}$ choices for $D_1$ and $D_2$ satisfying these conditions. Applying Proposition~\ref{prop:largedroplets} to $D_1$ and $D_2$ (which we may do since $h^*(D_i) \leq 1-\eps$ for $i \in \{1,2\}$), and using the BK inequality, it follows that
\begin{equation}\label{eq:lastline}
\P_p\big( [A] = \Z_n^2 \big) \leq n^2 \cdot p^{-O(1)} \cdot \exp\Bigg( - \frac{(1 - 8\eps)}{4p} \bigg( \log \frac{1}{p} \bigg)^2 \Bigg) \leq n^{-\eps}
\end{equation}
if $\eps > 0$ is sufficiently small. This complete the proof of the theorem.
\end{proof}

\section{Further discussion and open problems}\label{sec:open}

\subsection{The modified Duarte model}\label{sec:modifiedDuarte}

The \emph{modified Duarte model} is the monotone cellular automaton whose update family is
\[
\D^{(m)} := \Big\{ \big\{ (-1,0),(0,-1) \big\}, \big\{ (1,0),(0,-1) \big\} \Big\}.
\]
Thus, the modified Duarte model comprises two of the three rules of the (original) Duarte model, has the same stable set, and is also critical and unbalanced with difficulty 1. An interesting feature of the modified Duarte model is its size: it is formed of only two update rules, which is the minimum of any critical update family. The following theorem is the first sharp threshold result for a critical two-dimensional family that is minimal in this sense.

\begin{theorem}\label{thm:modified}
\[
p_c\big(\Z_n^2,\D^{(m)}\big) = \left(\frac{1}{4}+o(1)\right)\frac{(\log\log n)^2}{\log n}
\]
as $n\to\infty$.
\end{theorem}

The proof of Theorem~\ref{thm:modified} follows that of Theorem~\ref{thm:Duarte} almost exactly. The only differences are that the absence of the rule $\big\{ (0,-1),(0,1) \big\}$ from $\D^{(m)}$ means that, in order for a droplet to grow vertically, there must be an element of $A$ in every row, rather than just every alternate row. This leads to some small simplifications in Section~\ref{sec:crossings} and a gain of a factor of $2$ in the exponent in the bound~\eqref{eq:delta} in Lemma~\ref{lem:delta}, and some similarly minor simplifications in the upper bound.

\subsection{Related two-dimensional models}

In two dimensions, sharp thresholds are now known for the 2-neighbour model~\cite{Hol}, more generally for so-called symmetric balanced threshold models\footnote{That is, models formed by the $r$-element subsets of a centrally symmetric star subset of $\Z^2\setminus\{\0\}$, in the cases where such models are critical and balanced. (Here, `star' means that if $x$ is in the set then the whole of $(\0,x]\cap\Z^2$ is in the set.)}~\cite{DH}, for a single unbalanced non-drift model~\cite{DE}, and for the Duarte model, but remain open in all other cases.\footnote{Strictly speaking, sharp thresholds are also known for some minor variants of these models; in particular, for the modified Duarte model (see above), the modified and `$k$-cross' analogues of the 2-neighbour model~\cite{Hol,HLR}, and a single class of unbalanced non-drift models~\cite{DE}. However, the proof of the sharp threshold for each of these models follows via simple modifications of the proof above, and of the proofs in~\cite{DE,Hol}, respectively.} It might be possible that, using a combination of the techniques from~\cite{Hol,DH} for balanced models, those from~\cite{DE} for unbalanced models without drift, and those introduced in~\cite{BDMS} and the present paper for unbalanced models with drift, one could determine the sharp threshold for any critical family $\U$ whose update rules are contained in the axes (i.e. such that for all $X\in\U$ and for all $(a,b)\in X$, we have $ab=0$). Nevertheless, we expect the following problem to be hard.

\begin{problem}\label{prob:2dsharp}
Determine the sharp threshold for any critical family whose update rules are contained in the axes.
\end{problem}

\subsection{Higher dimensions}

The study of monotone cellular automata in higher dimensions is notoriously difficult. In $\Z^d$ for $d\geq 3$, the only models for which sharp thresholds are known are the $r$-neighbour bootstrap percolation models~\cite{BBM3d,BBDM}, for each $2\leq r\leq d$. These $r$-neighbour models aside, even coarse thresholds (that is, thresholds up to a constant factor) are only known for a certain family of symmetric three-dimensional threshold models, whose rules are contained in the axes~\cite{EF}.

The analogue of Problem~\ref{prob:2dsharp} in dimensions $d\geq 3$ is likely to be out of reach at present, but it may be possible to make progress if `sharp threshold' is replaced by `coarse threshold'. To state the problem formally, we need to say what we mean by `critical' in higher dimensions. The following definition was recently proposed by the authors in~\cite{BDMS}.

Fix an integer $d\geq 2$ and let $\U$ be a $d$-dimensional update family (that is, let $\U$ be a finite collection of finite subsets of $\Z^d \setminus \{\0\}$). Define the stable set $\stab = \stab(\U)$ analogously to how it is defined in two dimensions:
\[
\stab := \big\{ u\in S^{d-1} \,:\, [\H_u^d]=\H_u^d \big\},
\]
where
\[
\H_u^d := \big\{ x\in\Z^d \,:\, \<x,u\> < 0 \big\}
\]
is the discrete half-space in $\Z^d$ with normal $u\in S^{d-1}$. Let $\sigma^{d-1}$ denote the spherical measure on $S^{d-1}$.

\begin{definition}
A $d$-dimensional update family is:
\begin{enumerate}
\item \emph{subcritical} if $\sigma^{d-1}(\stab\cap C)>0$ for every hemisphere $C\subset S^{d-1}$;
\item \emph{critical} if there exists a hemisphere $C\subset S^{d-1}$ such that $\sigma^{d-1}(\stab\cap C)=0$ and if $\stab\cap C\neq\emptyset$ for every open hemisphere $C\subset S^{d-1}$;
\item \emph{supercritical} if $\stab\cap C=\emptyset$ for some open hemisphere $C\subset S^{d-1}$.
\end{enumerate}
\end{definition}

\begin{problem}
For each $d \geq 3$, determine the coarse threshold for any $d$-dimensional critical family whose update rules are contained in the axes.
\end{problem}

This question is already likely to be very difficult, so as a first step one might restrict to the case $d=3$ or to update rules contained in the set of nearest neighbours of the origin.

\bibliographystyle{amsplain}
\bibliography{bprefs}

\end{document}